\documentclass[oneside,english]{siamart220329}

\usepackage[T1]{fontenc}
\usepackage[latin9]{inputenc}
\usepackage{xcolor}
\usepackage{array}
\usepackage{verbatim}
\usepackage{prettyref}
\usepackage{calc}
\usepackage{enumitem}
\usepackage{amsbsy}
\usepackage{amstext}
\usepackage{amssymb}
\usepackage{graphicx}

\makeatletter

\providecommand{\tabularnewline}{\\}

\numberwithin{equation}{section}
\numberwithin{figure}{section}
\theoremstyle{plain}

\usepackage{lmodern}
\usepackage{xcolor}
\usepackage[T1]{fontenc}
\usepackage{svg}
\definecolor{green}{rgb}{0,0.5,0.0} 
\hypersetup{
    colorlinks,
    citecolor=green,
    filecolor=green,
    linkcolor=green,
    urlcolor=green
}

\DeclareMathSizes{8}{8}{6}{5} 
\DeclareMathSizes{7}{7}{5}{4} 
\DeclareMathSizes{6}{6}{4}{3} 
\DeclareMathSizes{5}{5}{3}{2} 

\usepackage{cite}
\let\prettyref\Cref
\Crefname{algocf}{Algorithm}{Algorithms}
\Crefname{ALC@unique}{Line}{Lines}
\crefname{line}{line}{lines}

\usepackage[noend]{algpseudocode}
\usepackage{dashrule}

\makeatletter
\def\ALG@special@indent{%
    \ifdim\ALG@thistlm=0pt\relax
        \hskip-\leftmargin
    \else
        \hskip\ALG@thistlm
    \fi
}
\newcommand{\StateO}[1]{\item[]\noindent\ALG@special@indent #1}
\makeatother

\renewcommand{\arraystretch}{1.25}

\makeatother

\usepackage{babel}

\providecommand{\theoremname}{Theorem}

\begin{document}
\title{Complexity-optimal and Parameter-Free First-Order Methods for Finding
Stationary Points of Composite Optimization Problems\thanks{\textbf{Funding}: The work of this author has been supported by the
Exascale Computing Project (17-SC-20-SC), a collaborative effort of
the U.S. Department of Energy Office of Science and the National Nuclear
Security Administration.\protect \\
\textbf{Versions}: v1.0 (May 25, 2022), v2.0 (February 11, 2023),
v3.0 (September 17, 2023), v4.0 (February 9, 2024)}}
\author{Weiwei Kong\thanks{Work done at the Computer Science and Mathematics Division, Oak Ridge
National Laboratory, Oak Ridge, TN, 37830. \texttt{Email: wwkong92@gmail.com}}}

\maketitle
\global\long\def\tx{\tilde{x}}%
\global\long\def\rn{\mathbb{R}^{n}}%
\global\long\def\R{\mathbb{R}}%
\global\long\def\r{\mathbb{R}}%
\global\long\def\n{\mathbb{N}}%
\global\long\def\c{\mathbb{C}}%
\global\long\def\pt{\mathbb{\partial}}%
\global\long\def\lam{\lambda}%
\global\long\def\argmin{\operatorname*{argmin}}%
\global\long\def\Argmin{\operatorname*{Argmin}}%
\global\long\def\argmax{\operatorname*{argmax}}%
\global\long\def\dom{\operatorname*{dom}}%
\global\long\def\ri{\operatorname*{ri}}%
\global\long\def\inner#1#2{\langle#1,#2\rangle}%
\global\long\def\trc{\operatorname*{tr}}%
\global\long\def\cConv{\overline{{\rm Conv}}\ }%
\global\long\def\intr{\operatorname*{int}}%

\begin{abstract}
This paper develops and analyzes an accelerated proximal descent method
for finding stationary points of nonconvex composite optimization
problems. The objective function is of the form $f+h$ where $h$
is a proper closed convex function, $f$ is a differentiable function
on the domain of $h$, and $\nabla f$ is Lipschitz continuous on
the domain of $h$. The main advantage of this method is that it is
``parameter-free'' in the sense that it does not require knowledge
of the Lipschitz constant of $\nabla f$ or of any global topological
properties of $f$. It is shown that the proposed method can obtain
an $\varepsilon$-approximate stationary point with iteration complexity
bounds that are optimal, up to logarithmic terms over $\varepsilon$,
in both the convex and nonconvex settings. Some discussion is also
given about how the proposed method can be leveraged in other existing
optimization frameworks, such as min-max smoothing and penalty frameworks
for constrained programming, to create more specialized parameter-free
methods. Finally, numerical experiments are presented to support the
practical viability of the method.
\end{abstract}



\section{Introduction}

\label{sec:intro}

Consider the nonsmooth composite optimization problem
\begin{equation}
\phi_{*}=\min_{z\in\rn}\left\{ \phi(z):=f(z)+h(z)\right\} \label{eq:main_prb}
\end{equation}
where $h:\rn\mapsto(\infty,\infty]$ is a proper closed convex function,
$f$ is a (possibly nonconvex) continuously differentiable function
on an open set containing the domain of $h$ (denoted as $\dom h$),
and $\nabla f$ is Lipschitz continuous. It is well known that the
above assumption on $f$ implies the existence of positive scalars
$m$ and $M$ such that 
\begin{equation}
-\frac{m}{2}\|x-x'\|^{2}\leq f(x)-f(x')-\left\langle \nabla f(x'),x-x'\right\rangle \leq\frac{M}{2}\|x-x'\|^{2}\label{eq:intro_curv}
\end{equation}
for every $x,x'\in\dom h$. The quantity $(m,M)$ is often called
a \emph{curvature pair }of $\phi$ (see, for example, \cite{liang2018doubly,liang2021fista}),
and the first inequality of \eqref{eq:intro_curv} is often called
\emph{weak-convexity} when $m>0$ (see, for example, \cite{drusvyatskiy2017proximal,davis2019stochastic}).

Recently, there has been a surge of interest in developing efficient
algorithms for finding $\varepsilon$-stationary points of \eqref{eq:main_prb},
which consist of a pair $(\bar{z},\bar{v})\in\dom h\times\rn$ satisfying
\textbf{
\begin{equation}
\bar{v}\in\nabla f(\bar{z})+\partial h(\bar{z}),\quad\|\bar{v}\|\leq\varepsilon.\label{eq:stationary_def}
\end{equation}
}While complexity-optimal algorithms exist for the case where both
$m$ and $M$ are known, a \emph{parameter-free} algorithm --- one
without knowledge of $(m,M)$ --- with optimal iteration complexity
remains elusive. 

Our goal in this paper is to develop, analyze, and extend a parameter-free
\emph{accelerated} \emph{proximal} \emph{descent} (PF.APD) algorithm
that obtains, up-to-logarithmic terms, optimal iteration complexities
regardless of the convexity of $f$. Roughly speaking, PF.APD generates
a sequence of iterates $\{(z_{k},m_{k})\}\subseteq\dom h\times\r_{++}$
which satisfies
\begin{equation}
z_{k+1}\approx\argmin_{z\in\dom h}\left\{ \frac{\phi(z)}{2m_{k+1}}+\frac{1}{2}\|z-z_{k}\|^{2}\right\} ,\quad\phi(z_{k+1})\leq\phi(z_{k}).\label{eq:approx_scheme}
\end{equation}
for every $k\geq0$. Notice that the first expression in \eqref{eq:approx_scheme}
is an inexact \emph{proximal} point update with stepsize $1/(2m_{k+1})$,
while the inequality in \eqref{eq:approx_scheme} implies $\{\phi(z_{k})\}$
is a \emph{descent} sequence. More precisely, the $(k+1)$-th iteration
of PF.APD is as follows:
\begin{center}
\vspace*{1em}%
{\fboxrule 0.5pt\fboxsep 6pt\fbox{\begin{minipage}[t]{0.85\columnwidth}%
\textbf{Iteration $k+1$:}
\begin{itemize}
\item[(i)] Given $\hat{m}\in\r_{++}$, find a \emph{proximal descent }point
$z_{k+1}\in\dom h$ in which there exists $\hat{u}\in\rn$ satisfying
\begin{align}
 & \hat{u}\in\nabla f(z_{k+1})+\pt\left(h+\hat{m}\|\cdot-z_{k}\|^{2}\right)(z_{k+1}),\label{eq:intro_incl}\\
 & \|\hat{u}+\hat{m}(z_{k}-z_{k+1})\|^{2}\leq2\theta\hat{m}\left[\phi(z_{k})-\phi(z_{k+1})\right],\label{eq:intro_invar1}\\
 & \|\hat{u}\|^{2}\leq2(\rho\hat{m})^{2}\|z_{k+1}-z_{k}\|^{2},\label{eq:intro_invar2}
\end{align}
for some $\theta>0$ and $\rho>0$.
\item[(ii)] If a key inequality fails during the execution of step~(i), change
$\hat{m}$ and try step (i) again. Else, set $m_{k+1}=\hat{m}$.
\end{itemize}
\end{minipage}}}\vspace*{1em}
\par\end{center}

To find $z_{k+1}$ in step (i) in the above outline, PF.APD specifically
applies a parameter-free \emph{accelerated} composite gradient (PF.ACG)
algorithm to the subproblem $\min_{z\in\dom h}\{\phi(z)/(2\hat{m})+\|z-z_{k}\|^{2}/2\}$
until a finite set of key descent inequalities holds. During the execution
of PF.ACG, several inequalities are also checked to ensure its convergence
(specifically the ones in \eqref{eq:acg_cvx_cond}), and execution
is halted if at least one of these inequalities does not hold. These
inequalities are always guaranteed to hold when $\hat{m}\geq m$ but
may fail to hold when $\hat{m}<m$.

It is worth mentioning that the main difficulties preventing the extension
of existing complexity-optimal methods to parameter-free ones is their
dependence on \emph{global} topological conditions that strongly depend
on the knowledge of $(m,M)$, e.g., \eqref{eq:intro_curv}, convexity
of $f$, or knowledge of the Lipschitz modulus of $\nabla f$. Hence,
one of the novelties of PF.APD is its ability to relax these conditions
to a finite set of \emph{local} topological conditions that only depend
on the generated sequence of iterates.\smallskip{}

\subsection{Literature Review}

To keep our notation concise, we will make use of\vspace{-1em}

\begin{equation}
\Delta_{0}:=\phi(z_{0})-\inf_{z\in\rn}\phi(z),\quad d_{0}:=\inf_{z_{*}\in\rn}\left\{ \|z_{0}-z_{*}\|:\phi(z_{*})=\inf_{z\in\rn}\phi(z)\right\} ,\label{eq:optimality_residuals}
\end{equation}
with the assumption that $\Delta_{0}<\infty$ but $d_{0}$ may be
infinite. Furthermore, we break our discussion between the convex
and nonconvex settings and between two types of methods:
\begin{enumerate}[left=0.5em]
\item[\emph{I}.] Algorithms that find $\hat{z}\in\dom h$ satisfying $\phi(\hat{z})-\inf_{z\in\rn}\phi(z)\leq\varepsilon$;
\item[\emph{II}.] Algorithms that find $\bar{z}\in\dom h$ satisfying ${\rm dist}(0,\nabla f(\bar{z})+\pt h(\bar{z}))\leq\varepsilon$.
\end{enumerate}
It is worth mentioning that complexity-optimal \emph{type-I} methods
are not necessarily complexity-optimal \emph{type-II} methods, as
noted in \cite{nesterov2012make}.

\subsubsection*{Convex Setting}

For this discussion, we assume $\phi$ to be convex. Paper \cite{nesterov1983}
presents the first complexity-optimal \emph{type-I }methods, under
the assumption that $\max\{m,M\}$ is known. Papers \cite{guminov2019accelerated,guminov2021combination,nesterov2020primal,nesterov2013gradient}
(resp. paper \cite{neumaier2016osga}) present parameter-free complexity-optimal
\emph{type-I} methods for the case of $h\equiv0$ (resp. $h$ being
the indicator of a closed convex set). Paper \cite{ahookhosh2018solving}
extends the method in \cite{neumaier2016osga} to another parameter-free
complexity-optimal \emph{type-I} method for general convex functions
$h$.

The regularized accelerated method described in \cite{nesterov2012make}
is one of the earliest nearly-optimal (up to logarithmic terms) \emph{type-II}
methods for the case of $h\equiv0$. However, its complexity is obtained
under the strong assumption that: (i) $\max\{m,M\}$ is known, (ii)
that there exists $z_{*}\in\dom h$ such that $\phi(z_{*})=\inf_{z\in\rn}\phi(z)$,
(iii) and that a lower bound for $d_{0}$ is known. Whether a parameter-free
complexity-optimal \emph{type-II} method exists in the convex setting
is still unknown.

\subsubsection*{Nonconvex Setting}

For this discussion, we assume $\phi$ to be nonconvex. One of the
most well-known parameter-free \emph{type-II} algorithms is the proximal
gradient descent (PGD) method with backtracking line search. In \cite{nesterov2013gradient},
it was shown that this method has a ${\cal O}(\varepsilon^{-2})$
\emph{type-II} complexity bound when $f$ is weakly-convex and a suboptimal
${\cal O}(\varepsilon^{-1})$ \emph{type-II} bound when $f$ is convex. 

One of the earliest accelerated \emph{type-II} methods is found in
\cite{ghadimi2016accelerated} under the assumption that $\dom h$
is bounded. Following this, paper \cite{ghadimi2019generalized} presented
a parameter-free extension of the method in \cite{ghadimi2016accelerated}
that handles Hölder continuous gradients of $f$. In a separate line
of research, \cite{liang2021fista} presented a \emph{type-II} accelerated
method whose main steps are variants of the (accelerated) FISTA algorithm
in \cite{beck2009fast} and assumes $\dom h$ is bounded. A variant
of this method, with improved iteration complexity bounds in the convex
setting, was examined in \cite{sim2020fista}. It is worth noting
that some of the methods in \cite{ghadimi2016accelerated,ghadimi2019generalized,liang2021fista,sim2020fista}
have optimal \emph{type-I} bounds when $f$ is convex but all the
methods have suboptimal \emph{type-II} bounds even when $f$ is convex.

Motivated by the developments in \cite{ghadimi2016accelerated}, other
papers, e.g., \cite{carmon2018accelerated,drusvyatskiy2019efficiency,paquette2017catalyst,li2015accelerated},
developed similar accelerated methods under different assumptions
on $f$ and $h$. Recently, \cite{kong2019complexity} proposed a
parameter-dependent accelerated inexact proximal point (AIPP) method
that has an optimal iteration complexity bound of ${\cal O}(\sqrt{Mm}\Delta_{0}/\varepsilon^{2})$
when $f$ is weakly convex but has no advantage when $f$ is convex.
The work in \cite{kong2020efficient} proposed an adaptive version
of AIPP where $(m,M)$ were estimated locally, but a lower bound for
$\max\{m,M\}$ was still required. A version of \cite{kong2019complexity}
in which the outer proximal point scheme is replaced with an accelerated
one was examined in \cite{liang2018doubly}, in which a moderately
worse iteration-complexity bound was established. 

\subsubsection*{Tangentially Related Works}

The developments in \cite{kong2021fista,kong2021accelerated,kong2019complexity}
strongly influenced and motivated the technical developments of both
PF.ACG and PF.APD. Since PF.APD shares strong similarities with AIPP
in \cite{kong2019complexity}, we mention one of the former's technical
improvements on the latter. To begin, note that AIPP is a double-loop
method that repeatedly calls an ACG-type method on a sequence of prox
subproblems to generate a sequence of \emph{outer} iterates $\{(z_{k},v_{k},\varepsilon_{k})\}$
(at the end of each ACG call) satisfying 
\begin{equation}
v_{k}\in\partial_{\varepsilon_{k}}\left(\frac{\phi}{2m}+\frac{1}{2}\|\cdot-z_{k-1}\|^{2}\right)(z_{k}),\quad\|v_{k}\|^{2}+2\varepsilon_{k}\leq\sigma^{2}\|v_{k}+z_{k-1}-z_{k}\|^{2},\label{eq:hpe}
\end{equation}
where $\sigma\in(0,1)$ and $\pt_{\varepsilon}\psi(x):=\{u\in\r^{n}:\psi(z')\geq\psi(z)+\inner u{z'-z}-\varepsilon,\quad\forall z'\in\r^{n}\}$.
An expensive refinement procedure, whose effectiveness strongly depends
on \eqref{eq:hpe} and knowledge of $\max\{m,M\}$, is then applied
to each $(z_{k},v_{k},\varepsilon_{k})$ to obtain $(\bar{z},\bar{v})$
satisfying the inclusion in \eqref{eq:stationary_def}. In contrast,
the iterates generated at every \emph{inner} iteration of PF.APD always
satisfy the inclusion in \eqref{eq:stationary_def}, for a different
choice of $\bar{v}$ (see \prettyref{lem:acg_specialization}), and,
consequently, the termination of PF.APD can be checked at \emph{every}
one of its inner iterations \emph{without} the need for an expensive
refinement procedure. It is worth mentioning those relative prox-stationarity
criteria, such as \eqref{eq:intro_invar2} and \eqref{eq:hpe}, were
previously analyzed in \cite{rockafellar1976augmented} and, more
recently, in \cite{monteiro2010complexity,alves2016regularized,marques2019iteration,monteiro2010convergence,monteiro2015hybrid,monteiro2016adaptive}.

We now make a brief comparison between PF.APD and two adaptive proximal
methods in the literature. First, compared to the redistributed prox-bundle
(RPB) method in \cite{hare2010redistributed}, both PF.APD and RPB
are double-loop methods consisting of (i) outer (or ``serious'')
iterations that consider prox-subproblems of the form in \eqref{eq:approx_scheme}
and some $\lam>0$ and (ii) inner (or ``null'') iterations that
consider composite subproblems of the form $\min_{y\in\rn}\{\Phi_{j,k}(y)+h(y)\}$
for the $k$-th subproblem and $j$-th iteration, until there is a
sufficient decrease in $\phi(z_{k})$. However, PF.APD chooses $\Phi_{j,k}$
to be a quadratic approximation of $\Phi_{k}$ centered on a specially
chosen point (see the update of $y_{k+1}$ in \prettyref{alg:pf_acg_line_search}),
while RPB chooses $\Phi_{j,k}$ to be the maximum of a different set
of quadratic approximations, which is generally more difficult to
minimize. Moreover, PF.APD uses values of $\nabla f(\cdot)$ and elements
of $\pt h(\cdot)$ in its construction of $\Phi_{j,k}$ whereas RPB
uses elements of the limiting subdifferential of $\phi$. 

Second, compared to the Catalyst Acceleration Framework (CAF) in \cite{paquette2017catalyst},
both PF.APD and CAF consider inexactly solving proximal subproblems
as in \eqref{eq:approx_scheme} using ACG subroutine and subproblem
termination conditions similar to \eqref{eq:gd_ineq1}--\eqref{eq:gd_ineq2}.
However, CAF obtains the inequality in \eqref{eq:approx_scheme} by
inexactly solving a second prox-subproblem (with a different prox
center) and applying an extra interpolation step. As a consequence,
CAF requires nearly double the work of PF.APD. Moreover, the line
search strategy (analogous to \prettyref{alg:pf_acg_line_search}
and \prettyref{alg:pf_apd_line_search}) employed by CAF in \cite[Algorithm 3]{paquette2017catalyst}
is static in that it prescribes a large number of ACG iterations,
whereas the line search strategy in PF.APD is dynamic in that it checks
a finite set of simple inequalities at each ACG iteration.

\subsection{Contributions}

Throughout, we refer to the two types of algorithms described in the
previous subsection. Given a starting point $z_{0}\in\dom h$ and
a tolerance $\varepsilon>0$, it is shown that PF.APD has the following
nice properties:
\begin{enumerate}[left=0.5em]
\item[(i)] for any $\hat{m}>0$, it always obtains a pair $(\bar{z},\bar{v})\in\dom h\times\rn$
satisfying \eqref{eq:stationary_def};
\item[(ii)] if $f$ is nonconvex, then it stops in ${\cal O}(\sqrt{mM\Delta_{0}}/\varepsilon^{2})$
resolvent evaluations\footnote{The notation $\tilde{O}(\cdot)$ ignores any terms that logarithmically
depend on the tolerance $\varepsilon$.};
\item[(iii)] if $f$ is convex, then it stops in ${\cal O}(\sqrt{M}\min\{\sqrt{\Delta_{0}}/\varepsilon,d_{0}/\sqrt{\varepsilon}\})$
resolvent evaluations;
\end{enumerate}
Both of the above complexity bounds are optimal (up to logarithmic
terms in) in terms of $\Delta_{0}$, $M$, $m$, and $\varepsilon$
(although suboptimal by a factor of $\sqrt{d_{0}}$ in the convex
case). Moreover, it appears to be the first time that a \emph{type-II}
parameter-free method has obtained such bounds\footnote{Compare this to the complexity-optimal methods in \cite{nesterov2012make}
and \cite{kong2019complexity} which require knowledge of $d_{0}$
and $(m,M)$, respectively.}. Improved iteration complexity bounds are also obtained when $d_{0}$
is known. Also, all of the above results are obtained under the mild
assumption that the optimal value in \eqref{eq:main_prb} is finite
and does not assume the boundedness of $\dom h$ (cf. \cite{liang2021fista,sim2020fista})
nor that an optimal solution of \eqref{eq:main_prb} exists.

For convenience, we compare in \prettyref{tab:compl_compare} the
best iteration complexity bounds of some of the parameter-free methods
listed in the previous subsection with two instances of PF.APD. For
shorthands, PGD is the adaptive proximal gradient descent method in
\cite{nesterov2013gradient}, UPF is the UPFAG method in \cite{ghadimi2019generalized},
ANCF is the ADAP-NC-FISTA method in \cite{liang2021fista}, VRF is
the VAR-FISTA method in \cite{sim2020fista}, and APD is as in \prettyref{alg:pf_apd}
in this paper with $m_{0}=1$.

\begin{table}[!tbh]
\begin{centering}
\begin{tabular}{>{\centering}p{2.2cm}|c|c|c}
\noalign{\vskip0.1em}
\textbf{\footnotesize{}Algorithm} & \textbf{\footnotesize{}$f$ convex} & \textbf{\footnotesize{}$f$ nonconvex} & \textbf{\footnotesize{}$D_{h}<\infty$}\tabularnewline[0.1em]
\hline 
\noalign{\vskip0.2em}
{\scriptsize{}PGD\cite{nesterov2013gradient}} & {\footnotesize{}${\cal O}\left(\frac{M^{3/2}d_{0}}{\varepsilon}\right)$} & {\footnotesize{}${\cal O}\left(\frac{M^{2}\Delta_{0}}{\varepsilon^{2}}\right)$} & {\scriptsize{}No}\tabularnewline[0.2em]
\noalign{\vskip0.2em}
{\scriptsize{}UPF\cite{ghadimi2019generalized}} & {\scriptsize{}N/A} & {\footnotesize{}${\cal O}\left(\frac{M\Delta_{0}}{\varepsilon^{2}}\right)$} & {\scriptsize{}No}\tabularnewline[0.2em]
\noalign{\vskip0.2em}
{\scriptsize{}ANCF\cite{liang2021fista}} & {\footnotesize{}${\cal O}\left(\frac{M^{2/3}[\Delta_{0}^{1/3}+d_{0}^{2/3}]}{\varepsilon^{2/3}}+\frac{MD_{h}}{\varepsilon}\right)$} & {\footnotesize{}${\cal O}\left(mM^{2}\left[\frac{mD_{h}^{2}+\Delta_{0}}{\varepsilon^{2}}\right]\right)$} & {\scriptsize{}Yes}\tabularnewline[0.2em]
\noalign{\vskip0.2em}
{\scriptsize{}VRF\cite{sim2020fista}} & {\footnotesize{}${\cal O}\left(\frac{M^{2/3}[\Delta_{0}^{1/3}+D_{h}^{2/3}]}{\varepsilon^{2/3}}\right)$} & {\footnotesize{}${\cal O}\left(mM^{2}D_{h}^{2}\left[\frac{1+m^{2}}{\varepsilon^{2}}\right]\right)$} & {\scriptsize{}Yes}\tabularnewline[0.2em]
\noalign{\vskip0.2em}
\textbf{\scriptsize{}APD}{\footnotesize{}\medskip{}
} & {\footnotesize{}${\cal O}\left(\sqrt{M}\left[\min\left\{ \frac{\sqrt{\Delta_{0}}}{\varepsilon},\frac{d_{0}}{\sqrt{\varepsilon}}\right\} \right]\right)$} & {\footnotesize{}${\cal O}\left(\frac{\sqrt{mM}\Delta_{0}}{\varepsilon^{2}}\right)$} & {\scriptsize{}No}\tabularnewline[0.2em]
\hline 
\noalign{\vskip0.2em}
\textbf{\footnotesize{}Known Lower Bounds} & {\footnotesize{}$\Omega\left(\sqrt{M}\left[\min\left\{ \frac{\sqrt{\Delta_{0}}}{\varepsilon},\sqrt{\frac{d_{0}}{\varepsilon}}\right\} \right]\right)$} & {\footnotesize{}$\Omega\left(\frac{\sqrt{mM}\Delta_{0}}{\varepsilon^{2}}\right)$} & {\small{}-}\tabularnewline[0.2em]
\end{tabular}
\par\end{centering}
\caption{Lower bounds and iteration-complexity bounds of various parameter-free
\emph{type-II} composite optimization algorithms for finding $\varepsilon$-stationary
points as in \eqref{eq:stationary_def}. The scalar $D_{h}$ denotes
the diameter of $\protect\dom h$ and it is assumed that $d_{0}$,
$\Delta_{0}$, $m$, and $M$ are not known but $M$ is greater than
or equal to $m$ for the listed algorithms. The lower bounds for the
convex (resp. nonconvex) case can be found in \cite[Theorem 1]{carmon2021lower}
(resp. \cite[Theorem 4.5]{zhou2019lower}). \label{tab:compl_compare}}
\end{table}

Notice that the analysis for UPFAG does not include an iteration complexity
bound for finding stationary points when $f$ is convex, while ANCF
and VRF suffer from the requirement that $\dom h$ must be bounded.
Moreover, up until this point, PGD was the only parameter-free \emph{type-II}
algorithm with an established iteration complexity bound for the unbounded
case when $f$ is convex. None of the parameter-free methods before
this work, in the nonconvex setting, could obtain the optimal complexity
bound in \cite{kong2019complexity}.

In addition to the development of PF.APD, some details are given regarding
how PF.APD could be used in other existing optimization frameworks,
including min-max smoothing and penalty frameworks for constrained
optimization. The main advantages of these resulting frameworks are
that (i) they are parameter-free and (ii) they have improved complexities
when $f$ in \eqref{eq:main_prb} is convex, without requiring any
adjustments to their inputs.

Finally, numerical experiments are given to support the practical
efficiency of PF.ADP on some randomly generated problem instances.
These experiments specifically show that PF.APD consistently outperforms
several existing parameter-free methods in practice.

\subsection{Organization}

\prettyref{sec:background} presents background material. \prettyref{sec:pf_algs}
presents PF.ACG, PF.APD, and their iteration complexity bounds. \prettyref{sec:tech_proofs}
gives the proofs of several important technical results. \prettyref{sec:extensions}
describes how PF.APD can be used in existing optimization frameworks.
\prettyref{sec:numerical} presents some numerical experiments. \prettyref{sec:concluding}
gives some concluding remarks. Several technical appendices follow
after the above sections.\smallskip{}

\subsection{Notation and Basic Definitions}

$\r_{+}$ and $\r_{++}$ denote the set of nonnegative and positive
real numbers, respectively. $\r^{n}$ denotes an $n$-dimensional
Euclidean space with inner product and norm denoted by $\langle\cdot,\cdot\rangle$
and $\|\cdot\|$, respectively. ${\rm dist}(x,X)$ denotes the Euclidean
distance of a point $x$ to a set $X$. For any $t>0$, we denote
$\log^{+1}(t):=\max\{\log t,1\}$. For a function $h:\r^{n}\to(-\infty,\infty]$
we denote $\dom h:=\{x\in\r^{n}:h(x)<+\infty\}$ to be the domain
of $h$. Moreover, $h$ is considered proper if $\dom h\ne\emptyset$.
The set of all lower semi-continuous proper convex functions defined
in $\r^{n}$ is denoted by $\cConv(\rn)$. The convex subdifferential
of a proper function $h:\r^{n}\to(-\infty,\infty]$ is given by 
\begin{equation}
\partial h(z):=\{u\in\r^{n}:h(z')\geq h(z)+\inner u{z'-z},\quad\forall z'\in\r^{n}\}\label{def:epsSubdiff}
\end{equation}
for every $z\in\r^{n}$. If $\psi$ is a real-valued function which
is differentiable at $\bar{z}\in\r^{n}$, then its affine/linear approximation
$\ell_{\psi}(\cdot,\bar{z})$ at $\bar{z}$ is given by 
\begin{equation}
\ell_{\psi}(z;\bar{z}):=\psi(\bar{z})+\inner{\nabla\psi(\bar{z})}{z-\bar{z}}\quad\forall z\in\r^{n}.\label{eq:defell}
\end{equation}

\section{Background}

\label{sec:background}

This section gives some necessary background for presenting PF.ACG
and PF.APD. More specifically, \prettyref{subsec:prb_descr} describes
and comments on the problem of interest, while \prettyref{subsec:descent_scheme}
presents a general proximal descent scheme which serves as a template
for PF.APD.

\subsection{Problem of Interest}

\label{subsec:prb_descr}

To reiterate, we are interested in the following composite optimization
problem:
\begin{center}
\vspace*{1em}%
{\fboxrule 0.5pt\fboxsep 6pt\fbox{\begin{minipage}[t]{0.85\columnwidth}%
\textbf{Problem }$\boldsymbol{{\cal CO}}$: Given $\varepsilon\in\r_{++}$
and a function $\phi=f+h$ satisfying:
\begin{enumerate}[start=1,label={{\color{green}\tt$\langle$A\arabic*$\rangle$}}]
\item $h\in\cConv(\rn)$ and the resolvent $(\lam\pt h+{\rm id})^{-1}$
is easy to compute for any $\lam>0$\label{asm:a1},
\item $f$ is continuously differentiable on an open set $\Omega\supseteq\dom h$,
and $\nabla f$ is ${\cal M}$-Lipschitz continuous on $\dom h$ for
some ${\cal M}\in\r_{++}$\label{asm:a2},
\item $\phi_{*}=\inf_{z\in\rn}\phi(z)>-\infty$\label{asm:a3},
\end{enumerate}
find a pair $(\bar{z},\bar{v})\in\dom h\times\rn$ satisfying \eqref{eq:stationary_def}.%
\end{minipage}}}\vspace*{1em}
\par\end{center}

Of the three above assumptions, only \ref{asm:a1} is a necessary
condition that is used to ensure PF.APD is well-defined. Assumptions
\ref{asm:a2}--\ref{asm:a3}, on the other hand, are sufficient conditions
that are used to show that PF.APD stops in a finite number of iterations.
It is possible to replace assumption \ref{asm:a2} with more general
smoothness conditions (e.g., Hölder continuity \cite{nesterov2015universal,ghadimi2019generalized})
at the cost of a possibly more complicated analysis. It is known\footnote{The proof of the forward direction is well-known (see, for example,
\cite{nesterov2018lectures,beck2017first}) while the proof of the
reverse direction can be found, for example, in \cite[Proposition 2.1.55]{kong2021accelerated}.
For the special case where $f$ is convex and real-valued, the proof
of the reverse direction can be found, for example, in \cite[Theorem 18.15]{bauschke2011convex}
and \cite[2.1.5]{nesterov2003introductory}.} that assumption \ref{asm:a2} holds if and only if 
\begin{equation}
|f(z)-\ell_{f}(z;z')|\leq\frac{{\cal M}}{2}\|z-z'\|^{2},\quad\forall z,z'\in\dom h,\label{eq:descent_bd}
\end{equation}
which implies $({\cal M},{\cal M})$ is a curvature pair of $\phi$.

We now comment on criterion \eqref{eq:stationary_def}. First, it
is related to the directional derivative of $\phi$:
\begin{align*}
\min_{\|d\|=1}\phi'(z;d) & =\min_{\|d\|=1}\max_{\zeta\in\pt h(z)}\left\langle \nabla f(z)+\zeta,d\right\rangle =\max_{\zeta\in\pt h(z)}\min_{\|d\|=1}\left\langle \nabla f(z)+\zeta,d\right\rangle \\
 & =-\min_{\zeta\in\pt h(z)}\|\nabla f(z)+\zeta\|=-{\rm dist}(0,\nabla f(z)+\partial h(z)).
\end{align*}
Consequently, if $\bar{z}\in\dom h$ is a local minimum of $\phi$
then $\min_{\|d\|=1}\phi'(\bar{z};d)\geq0$ and the above relation
implies that \eqref{eq:stationary_def} holds with $\varepsilon=0$.
That is, \eqref{eq:stationary_def} is a necessary condition for local
optimality of a point $\bar{z}\in\dom h$. Second, when $f$ is convex
then \eqref{eq:stationary_def} with $\varepsilon=0$ implies that
$0\in\nabla f(\bar{z})+\pt h(\bar{z})=\partial\phi(\bar{z})$ and
$\bar{z}$ is a global minimum. Given the first comment, \eqref{eq:stationary_def}
is equivalent to global optimality of a point $\bar{z}\in\dom h$
when $f$ is convex. 

\subsection{General Proximal Descent Scheme}

\label{subsec:descent_scheme}

Our interest in this subsection is the general proximal descent scheme
in \prettyref{alg:gpds}, which follows the ideas in \eqref{eq:intro_incl}--\eqref{eq:intro_invar2}.
Its iteration scheme serves as a template for the PF.APD presented
in \prettyref{subsec:apd}.

\begin{algorithm}[htb]
\caption{\small General Proximal Descent Scheme}
\label{alg:gpds}
\begin{algorithmic}[1]

\StateO \texttt{\small{}Data}{\small{}: $(f,h)$ as in \ref{asm:a1}--\ref{asm:a3},
$z_{0}\in\dom h$;}{\small\par}

\StateO \texttt{\small{}Parameters:}{\small{} $\theta,\rho\in\r_{+}$;}{\small\par}

\For{$k\gets0,1,\ldots$}

\State \textbf{find} $(z_{k+1},u_{k+1})\in\dom h\times\rn$ and $m_{k+1}\in\r_{++}$
satisfying\vspace*{-0.75em}
\begin{align}
 & u_{k+1}\in\nabla f(z_{k+1})+2m_{k+1}(z_{k+1}-z_{k})+\partial h(z_{k+1}),\label{eq:gd_incl}\\
 & \|u_{k+1}+2m_{k+1}(z_{k}-z_{k+1})\|^{2}\leq2\theta m_{k+1}\left[\phi(z_{k})-\phi(z_{k+1})\right],\label{eq:gd_ineq1}\\
 & \|u_{k+1}\|^{2}\leq2(\rho m_{k+1})^{2}\|z_{k+1}-z_{k}\|^{2}.\label{eq:gd_ineq2}
\end{align}
\vspace{-1.5em}

\EndFor 

\end{algorithmic}
\end{algorithm}

Before presenting the properties of \prettyref{alg:gpds}, let us
comment on its steps. First, \eqref{eq:gd_incl}--\eqref{eq:gd_ineq2}
are analogous to \eqref{eq:intro_incl}--\eqref{eq:intro_invar2}
because of assumption \ref{asm:a1}. Second, if $f+m_{k+1}\|\cdot\|^{2}$
is convex and $u_{k+1}=0$ then \eqref{eq:gd_incl} implies that 
\[
z_{k+1}=\argmin_{z\in\dom h}\left\{ \frac{\phi(z)}{2m_{k+1}}+\frac{1}{2}\|z-z_{k+1}\|^{2}\right\} ,
\]
which is a proximal point update with stepsize $1/(2m_{k+1})$. Third,
\eqref{eq:gd_ineq1} implies that \prettyref{alg:gpds} is a descent
scheme, i.e., $\phi(z_{k+1})\leq\phi(z_{k})$ for $k\geq0$. Hence,
in view of the second comment, this justifies its qualifier as a ``proximal
descent'' scheme. 

It is also worth mentioning that \eqref{eq:gd_ineq1}--\eqref{eq:gd_ineq2}
are similar to conditions in existing literature. More specifically,
a version of \eqref{eq:gd_ineq1} can be found in the descent scheme
of \cite{kong2020efficient}, while an inequality similar to \eqref{eq:gd_ineq2}
can be found in the GIPP framework of \cite{kong2019complexity} with
$\sigma=2(\rho m_{k+1})^{2}$, $\tilde{\varepsilon}=0$, and $v_{k+1}=u_{k+1}/m_{k+1}$.
However, the addition of condition \eqref{eq:gd_incl} appears to
be new.

We now present the most important properties of \prettyref{alg:gpds}.
The first result supports the importance of conditions \eqref{eq:gd_incl}--\eqref{eq:gd_ineq1}.
\begin{lemma}
\label{lem:basic_gd_compl}Given $z_{0}\in X$, let $\{(z_{k+1},u_{k+1})\}_{k\geq0}$
denote a sequence of iterates satisfying \eqref{eq:gd_incl}--\eqref{eq:gd_ineq1}.
Moreover, let $\Delta_{0}$ be as in \eqref{eq:optimality_residuals},
and define 
\[
v_{k+1}:=u_{k+1}+2m_{k+1}(z_{k}-z_{k+1}),\quad\Lambda_{k+1}:=\sum_{j=0}^{k}\frac{1}{m_{j+1}},\quad\forall k\geq0.
\]
Then, for every $k\geq0$,
\end{lemma}

\begin{itemize}
\item[(a)] $v_{k+1}\in\nabla f(z_{k+1})+\partial h(z_{k+1})$;
\item[(b)] $\min\limits _{0\leq j\leq k}\|v_{j+1}\|^{2}\leq2\theta\Delta_{0}\Lambda_{k+1}^{-1}$.
\end{itemize}
\begin{proof}
(a) This follows immediately from \eqref{eq:gd_incl} and the definition
of $v_{k+1}$.

(b) Summing up both sides of \eqref{eq:gd_ineq1} from $0$ to $k$,
the definition of $v_{k+1}$, and the definition of $\phi_{*}$, we
have that
\begin{align*}
\Lambda_{k+1}\min_{0\leq j\leq k}\|v_{j+1}\|^{2} & \leq\sum_{j=0}^{k}\frac{\|v_{j+1}\|^{2}}{m_{j+1}}\overset{\eqref{eq:gd_ineq1}}{\leq}2\theta\sum_{j=0}^{k}\left[\phi(z_{j})-\phi(z_{j+1})\right]\\
 & =2\theta\left[\phi(z_{0})-\phi(z_{k+1})\right]\leq2\theta\left[\phi(z_{0})-\phi_{*}\right]=2\theta\Delta_{0}.
\end{align*}
\end{proof}
Notice that \prettyref{lem:basic_gd_compl}(b) implies that if $\lim_{k\to\infty}\Lambda_{k+1}\to\infty$
then we have that $\lim_{k\to\infty}\min_{j\leq k}\|v_{j+1}\|\to0$.
Moreover, if $\sup_{k\geq0}m_{k+1}<\infty$ then for any $\varepsilon>0$,
there exists some finite $j\geq0$ such that $\|v_{j+1}\|\leq\varepsilon$.

The next result shows that if $m_{k+1}$ is bounded relative to the
global topology of $f$, and conditions \eqref{eq:gd_incl}--\eqref{eq:gd_ineq2}
hold, then a more refined bound of $\min_{j\leq k}\|v_{j+1}\|$ can
be obtained. To keep the notation concise, we make use of the following
quantity: 
\begin{equation}
R_{\tau}(\hat{z}):=\inf_{z\in\rn}\left\{ R_{\tau}(z,\hat{z}):=\frac{\phi(z)-\phi_{*}}{\tau}+\frac{1}{2}\|z-\hat{z}\|^{2}\right\} .\label{eq:R_phi_def}
\end{equation}
It is easy to see that $R_{\tau}(z')$ is the Moreau envelope of $\phi/\tau$
at $z'$ shifted by $-\phi_{*}/\tau$.
\begin{lemma}
\label{lem:gd_compl_bd}Given $z_{0}\in X$, let $\{(v_{j+1},z_{j+1},\Lambda_{j+1})\}_{j\geq0}$
be as in \prettyref{lem:basic_gd_compl} and $k\geq0$ be fixed. Moreover,
suppose \eqref{eq:gd_ineq2} holds and that there exists $\tilde{m}>0$
such that $f+\tilde{m}\|\cdot\|^{2}/2$ is convex. If $\max_{0\leq j\leq k}m_{j+1}\leq\nu\tilde{m}$
for some $\nu\in(0,\rho^{-2}]$, then
\begin{equation}
\phi(z_{k+1})+m_{k+1}\left(1-\rho^{2}\nu\right)\|z_{k+1}-z_{k}\|^{2}\leq\inf_{z\in\rn}\left\{ \phi(z)+m_{k+1}\|z-z_{k}\|^{2}\right\} .\label{eq:prox_bd}
\end{equation}
Furthermore, if $k\geq1$ then it holds that
\begin{equation}
\min_{1\leq j\leq k}\|v_{j+1}\|^{2}\leq\frac{4\theta m_{1}}{\Lambda_{k+1}-m_{1}^{-1}}\left[R_{2m_{1}}(z_{0})-\left(\frac{1-\rho^{2}\nu}{2}\right)\|z_{1}-z_{0}\|^{2}\right].\label{eq:vj_ref_bd}
\end{equation}
\end{lemma}

\begin{proof}
Using the assumption that $m_{k+1}\geq\tilde{m}$ and \eqref{eq:gd_incl},
we have that $f(\cdot)+m_{k+1}\|\cdot-z_{k}\|^{2}$ is $\tilde{m}$-strongly
convex and, hence,
\begin{align}
u_{k+1} & \in\nabla f(z_{k+1})+2m_{k+1}(z_{k+1}-z_{k})+\partial h(z_{k+1})\nonumber \\
 & =\nabla f(z_{k+1})-\tilde{m}(z_{k+1}-z_{k+1})+2m_{k+1}(z_{k+1}-z_{k})+\partial h(z_{k+1})\nonumber \\
 & =\partial\left(\phi-\frac{\tilde{m}}{2}\|\cdot-z_{k+1}\|^{2}+m_{k+1}\|\cdot-z_{k}\|^{2}\right)(z_{k+1}).\label{eq:gd_tech_incl}
\end{align}
Using \eqref{eq:gd_tech_incl}, \eqref{eq:gd_ineq2}, and the bounds
$m_{k+1}\leq\nu\tilde{m}$ and $\langle a,b\rangle\geq-\nu\|a\|^{2}/(2m_{k+1})-m_{k+1}\|b\|^{2}/(2\nu)$
for any $a,b\in\rn$, it holds for any $z\in\rn$ that
\begin{align*}
 & \phi(z)+m_{k+1}\|z-z_{k}\|^{2}\\
 & \overset{\eqref{eq:gd_tech_incl}}{\geq}\phi(z_{k+1})+m_{k+1}\|z_{k}-z_{k+1}\|^{2}+\frac{\tilde{m}}{2}\|z-z_{k+1}\|^{2}+\left\langle u_{k+1},z-z_{k+1}\right\rangle \\
 & \geq\phi(z_{k+1})+m_{k+1}\|z_{k}-z_{k+1}\|^{2}-\frac{\nu}{2m_{k+1}}\|u_{k+1}\|^{2}+\frac{\tilde{m}-m_{k+1}/\nu}{2}\|z-z_{k+1}\|^{2}\\
 & \overset{\eqref{eq:gd_ineq2}}{\geq}\phi(z_{k+1})+m_{k+1}\left(1-\rho^{2}\nu\right)\|z_{k}-z_{k+1}\|^{2}+\frac{\tilde{m}-m_{k+1}/\nu}{2}\|z-z_{k+1}\|^{2}\\
 & \geq\phi(z_{k+1})+m_{k+1}\left(1-\rho^{2}\nu\right)\|z_{k}-z_{k+1}\|^{2}
\end{align*}
which implies \eqref{eq:prox_bd} as $z\in\rn$ was arbitrary. To
show \eqref{eq:vj_ref_bd}, we use \eqref{eq:prox_bd} at $k=0$,
the bound $\phi(z_{k+1})\geq\phi_{*}$, \eqref{eq:gd_ineq1}, and
the definition of $v_{k+1}$ to conclude that
\begin{align*}
 & R_{2m_{1}}(z_{0})-\left(\frac{1-\rho^{2}\nu}{2}\right)\|z_{1}-z_{0}\|^{2}\\
 & =\inf_{z\in\rn}\left\{ \frac{\phi(z)-\phi_{*}}{2m_{1}}+\frac{1}{2}\|z-z_{0}\|^{2}\right\} -\left(\frac{1-\rho^{2}\nu}{2}\right)\|z_{1}-z_{0}\|^{2}\overset{\eqref{eq:prox_bd}}{\geq}\frac{\phi(z_{1})-\phi_{*}}{2m_{1}}\\
 & \geq\frac{\phi(z_{1})-\phi(z_{k+1})}{2m_{1}}=\frac{\sum_{j=1}^{k}\left[\phi(z_{j})-\phi(z_{j+1})\right]}{2m_{1}}\\
 & \overset{\eqref{eq:gd_ineq1}}{\geq}\frac{1}{4\theta m_{1}}\sum_{j=1}^{k}\frac{\|v_{j+1}\|^{2}}{m_{j+1}}\geq\frac{\sum_{j=1}^{k}m_{j+1}^{-1}}{4\theta m_{1}}\left(\inf_{1\leq j\leq k}\|v_{j+1}\|^{2}\right)\\
 & =\frac{\Lambda_{k+1}-m_{1}^{-1}}{4\theta m_{1}}\left(\text{\ensuremath{\inf_{1\leq j\leq k}\|v_{j+1}\|^{2}}}\right).
\end{align*}
\end{proof}
Similar to the previous lemma, the above result also implies that
if $\lim_{k\to\infty}\Lambda_{k+1}\to\infty$ then we have $\lim_{k\to\infty}\min_{j\leq k}\|v_{j+1}\|\to0$.
However, it is more general in the sense that the rate of convergence
depends on $R_{2m_{1}}(z_{0})$ instead of $\Delta_{0}$, and the
former can be bounded as 
\begin{equation}
R_{2m_{1}}(z_{0})\leq\min\left\{ R_{2m_{1}}(z_{0},z_{0}),R_{2m_{1}}(z_{*},z_{0})\right\} \leq\frac{1}{2}\min\left\{ \frac{\Delta_{0}}{m_{1}},d_{0}^{2}\right\} ,\label{eq:R_bd}
\end{equation}
where $z_{*}$ is any optimal solution of \eqref{eq:main_prb} that
is the closest to $z_{0}$ and $(\Delta_{0},d_{0})$ are as in \eqref{eq:optimality_residuals}.
This fact will be important when we establish an iteration complexity
bound for PF.APD in the convex setting.

\section{Parameter-Free Algorithms}

\label{sec:pf_algs}

This section presents PF.ACG, PF.APD, and their iteration complexity
bounds. More specifically, \prettyref{subsec:acg} presents PF.ACG,
while \prettyref{subsec:apd} presents PF.APD.

It is also worth recalling that PF.APD is an implementation of the
general descent scheme of the previous section that repeatedly calls
PF.ACG to obtain a single iteration of the scheme mentioned above.

\subsection{PF.ACG Algorithm}

\label{subsec:acg}

Broadly speaking, PF.ACG is a modification of the well-known FISTA
\cite{florea2018accelerated,beck2009fast} algorithm for minimizing
$\mu$-strongly convex composite functions. Specifically, both PF.ACG
and FISTA consider the composite optimization problem 
\begin{equation}
\min_{x\in\rn}\left\{ \psi(x):=\psi^{s}(x)+\psi^{n}(x)\right\} \label{eq:acg_subprb}
\end{equation}
where $(\psi^{s},\psi^{n})$ satisfies the following assumptions:
\begin{enumerate}[start=1,label={{\color{green}\tt$\langle$B\arabic*$\rangle$}}]
\item $\psi^{n}\in\cConv(\rn)$ and the resolvent $(\lam\pt\psi^{n}+{\rm id})^{-1}$
is easy to compute for any $\lam>0$,\label{asm:b1}
\item $\psi^{s}$ is continuously differentiable on an open set $\Omega\supseteq\dom\psi^{n}$,
and $\nabla\psi^{s}$ is $L_{*}$-Lipschitz continuous on $\dom\psi^{n}$
for some $L_{*}\in\r_{++}$\label{asm:b2}.
\end{enumerate}
Similar to \eqref{eq:descent_bd}, note that \ref{asm:b2} implies
\begin{equation}
\left|\psi^{s}(x)-\ell_{\psi^{s}}(x;x')\right|\leq\frac{L_{*}}{2}\|x-x'\|^{2}\quad\forall x,x'\in\dom\psi^{n}.\label{eq:descent_acg}
\end{equation}

PF.ACG differs from FISTA in that it adds two stopping conditions
that help implement a single iteration of \prettyref{alg:gpds}. Specifically,
for a given function pair $(f,h)$ satisfying \ref{asm:a1}--\ref{asm:a2},
hyperparameters $(\sigma,\theta,\mu)\in\mathbb{R}_{++}^{3}$, and
an initial point $\hat{z}\in{\rm dom}\ h$, if PF.ACG is invoked with
\begin{equation}
\psi^{s}(\cdot)=\frac{f(\cdot)}{2\hat{m}}+\frac{1}{2}\|\cdot-\hat{z}\|^{2},\quad\psi^{n}(\cdot)=\frac{h(\cdot)}{2\hat{m}},\label{eq:psi_sn_vals}
\end{equation}
for some $\hat{m}>0$, then either (i) PF.ACG has found a pair $(y,u)$
satisfying conditions \eqref{eq:gd_incl}--\eqref{eq:gd_ineq2} with
$(z_{k+1},u_{k+1},m_{k+1},z_{k})=(y,u,m,\hat{z})$, or (ii) some local
$\mu$-strong convexity condition has failed, and the estimate of
$\mu$ or the function pair $(\psi^{s},\psi^{n})$ has to be changed.

We now present the details of PF.ACG and its key properties. To help
our discussion, we first give the complete pseudocode of PF.ACG through
\prettyref{alg:pf_acg_line_search} and \prettyref{alg:pf_acg}. More
specifically, \prettyref{alg:pf_acg_line_search} presents the accelerated
gradient FISTA update and (Lipschitz constant) line search strategy
used in PF.ACG, while \prettyref{alg:pf_acg} describes the other
steps of PF.ACG and how \prettyref{alg:pf_acg_line_search} is invoked. 

\begin{algorithm}[!h]
\caption{\small Line Search and Accelerated Gradient Step Subroutine}
\label{alg:pf_acg_line_search}
\begin{algorithmic}[1]

\StateO \texttt{\small{}Data}{\small{}: $(\psi^{s},\psi^{n})$ as
in \ref{asm:b1}--\ref{asm:b2}, $(\hat{y},\hat{x})\in\dom\psi^{n}\times\mathbb{R}^{n}$,
$\hat{A}\geq0$, $\mu\in\r_{++}$, $\hat{L}\in[\mu,\infty)$;}{\small\par}

\StateO \texttt{\small{}Hyper-parameters:}{\small{} $\beta\in(1,\infty)$;}{\small\par}

\StateO \texttt{\small{}Outputs}{\small{}: $(A,\tilde{x},y,x,L)\in\mathbb{R}_{+}\times\rn\times{\rm dom}\ \psi^{n}\times\mathbb{R}^{n}\times\mathbb{R}_{+}$
and function $q$;}{\small\par}

\State $\psi\gets\psi^{s}+\psi^{n}$

\For{$\ell\gets0,1,\ldots$}

\State $L\gets\hat{L}\beta^{\ell}$

\texttt{\textcolor{teal}{\small{}\hspace{-0.5em}$\triangleright$
Step 1:$\ $Accelerated gradient step.}}{\small\par}

\State $\xi\gets1+\mu\hat{A}\text{ and find \ensuremath{\hat{a}} satisfying }\hat{a}^{2}=\hat{\xi}(\hat{a}+\hat{A})/L$

\State $A\gets\hat{A}+\hat{a}$

\State $\tilde{x}\gets\frac{\hat{A}}{A}\hat{y}+\frac{\hat{a}}{A}\hat{x}$

\State $y\gets\argmin_{r\in\rn}\left\{ \ell_{\psi^{s}}(z;\tilde{x})+\psi^{n}(z)+\frac{L+\mu}{2}\|z-\tilde{x}\|^{2}\right\} $

\State $x\gets\hat{x}+\frac{\hat{a}}{1+A\mu}\left[L(y-\tilde{x})+\mu(y-\hat{x})\right]$

\texttt{\textcolor{teal}{\small{}\hspace{-0.5em}$\triangleright$
Step 2:$\ $Descent condition check.}}{\small\par}

\State \textbf{if} the inequality\vspace{-1em}
\begin{equation}
\psi^{s}(y)-\ell_{\psi^{s}}(y;\tilde{x})\leq\tfrac{L}{2}\|y-\tilde{x}\|^{2}\label{eq:acg_curv_cond}
\end{equation}
\vspace{-1.5em}

holds, then \textbf{return} $(A,\tilde{x},y,x,L)$

\EndFor 

\end{algorithmic}
\end{algorithm}

\begin{algorithm}[!h]
\caption{\small Parameter-Free Accelerated Composite Gradient (PF.ACG) Algorithm}
\label{alg:pf_acg}
\begin{algorithmic}[1]

\StateO \texttt{\small{}Data}{\small{}: $(\psi^{s},\psi^{n})$ as
in \ref{asm:b1}--\ref{asm:b2}, $y_{0}\in\dom\psi^{n}$, $\mu\in\r_{++}$,
$L_{0}\in[\mu,\infty)$;}{\small\par}

\StateO \texttt{\small{}Hyper-parameters:}{\small{} $\sigma\in\r_{++}$,
$\theta\in(2,\infty)$, $\beta\in(1,\infty)$;}{\small\par}

\StateO \texttt{\small{}Outputs}{\small{}: $(y_{j+1},u_{j+1},L_{j+1})\in\dom\psi^{n}\times\rn\times\r_{++}$;}{\small\par}

\State $(x_{0},A_{0})\gets(y_{0},0)$

\State $\psi(\cdot)\gets\psi^{s}(\cdot)+\psi^{n}(\cdot)$

\For{$j\gets0,1,\ldots$}

\texttt{\textcolor{teal}{\small{}\hspace{-0.5em}$\triangleright$
Step 1:$\ $Line search for $L_{j+1}$ and accelerated gradient step.}}{\small\par}

\State \textbf{call} \prettyref{alg:pf_acg_line_search} with data
$(\psi^{s},\psi^{n})$, $(\hat{y},\hat{x})\equiv(y_{j},x_{j})$, $\hat{A}\equiv A_{j}$,
$\hat{\xi}\equiv\xi_{j}$, $\mu$,

\StateO \quad{}$\hat{L}\equiv L_{j}$ and hyper-parameter $\beta$
to obtain $(A_{j+1},\tilde{x}_{j},y_{j+1},x_{j+1},L_{j+1})$

\texttt{\small{}\hspace{-0.5em}$\triangleright$ Step 2:$\ $``Bad''
termination check.}{\small\par}

\State $r_{j+1}\gets\nabla\psi^{s}(y_{j+1})-\nabla\psi^{s}(\tilde{x}_{j})+(L_{j+1}+\mu)(\tilde{x}_{j}-y_{j+1})$

\State \textbf{if }the inequalities\vspace{-1em}
\begin{equation}
\begin{aligned}\mu A_{j+1}\|y_{j+1}-\tilde{x}_{j}\|^{2} & \leq\|y_{j+1}-y_{0}\|^{2},\\
\psi(y_{0}) & \geq\psi(y_{j+1})+\left\langle r_{j+1},y_{0}-y_{j+1}\right\rangle ,
\end{aligned}
\label{eq:acg_cvx_cond}
\end{equation}
\vspace{-1em}

\emph{do} \emph{not} hold, then \textbf{return} $(y_{j+1},r_{j+1},L_{j+1})$

\texttt{\textcolor{teal}{\small{}\hspace{-0.5em}$\triangleright$
Step 3:$\ $``Good'' termination check.}}{\small\par}

\State \textbf{if }the inequalities\vspace{-1em}
\begin{equation}
\begin{aligned}\|r_{j+1}\|^{2} & \leq\sigma^{2}\|y_{j+1}-y_{0}\|^{2},\\
\|r_{j+1}+y_{0}-y_{j+1}\|^{2} & \leq\theta\left[\psi(y_{0})-\psi(y_{j+1})+\tfrac{1}{2}\|y_{j+1}-y_{0}\|^{2}\right],
\end{aligned}
\label{eq:acg_termination}
\end{equation}
\vspace{-1em}

hold, then \textbf{return} $(y_{j+1},r_{j+1},L_{j+1})$\label{ln:acg_good_stop}

\EndFor 

\end{algorithmic}
\end{algorithm}

We next present some key properties about \prettyref{alg:pf_acg}
and its iterates. As their proof is mostly technical, we moved it
to \prettyref{subsec:prf_acg_tech_props}.
\begin{lemma}
\label{lem:acg_tech_props} For every $j\geq0$,
\begin{itemize}
\item[(a)] $A_{j+1}\geq(1/L_{0})\prod_{i=1}^{j}[1+\sqrt{\mu/(2L_{i})}]$ and
\begin{equation}
L_{j}\leq L_{j+1}\leq\bar{L}:=\max\{1,\alpha L_{*}\}.\label{eq:Lbar_def}
\end{equation}
\item[(b)] $r_{j+1}\in\nabla\psi^{s}(y_{j+1})+\pt\psi^{n}(y_{j+1})$;
\item[(c)] if $\psi^{s}$ is $\mu$-strongly convex, then \eqref{eq:acg_cvx_cond}
holds;
\item[(d)] if \eqref{eq:acg_cvx_cond} holds and 
\begin{equation}
A_{j+1}\geq\frac{16\bar{L}^{2}}{\mu}\max\left\{ \frac{1}{\sigma^{2}},\frac{4\theta}{\theta-2}\right\} =:{\cal A}_{\mu,\bar{L}}(\sigma,\theta)\label{eq:calA_def}
\end{equation}
then \eqref{eq:acg_termination} holds.
\end{itemize}
\end{lemma}

We now give a complexity bound for \prettyref{alg:pf_acg} and a condition
for guaranteeing its successful termination.
\begin{proposition}
\label{prop:acg_alg_props}The following properties hold about \prettyref{alg:pf_acg}:
\end{proposition}

\begin{itemize}
\item[(a)] it stops in
\begin{equation}
\left\lceil 1+2\sqrt{\frac{2\bar{L}}{\mu}}\log^{1+}\left\{ \bar{L}{\cal A}_{\mu,\bar{L}}(\sigma,\theta)\right\} \right\rceil ,\label{eq:acg_complexity}
\end{equation}
where $\bar{L}$ and ${\cal A}_{\mu,\bar{L}}$ are as in \eqref{eq:Lbar_def}
and \eqref{eq:calA_def}, respectively.
\item[(b)] if $\psi^{s}$ is $\mu$-strongly convex, then it always terminates
in its Step~3 with a triple $(y_{j+1},u_{j+1},L_{j+1})$ satisfying
\eqref{eq:acg_termination} and $L_{0}\leq L_{j+1}\leq\bar{L}$.
\end{itemize}
\begin{proof}
(a) Let $J+1$ denote the quantity in \eqref{eq:acg_complexity} and
suppose \prettyref{alg:pf_acg} has not terminated at the end of iteration
$J+1$. Moreover, denote ${\cal A}:={\cal A}_{\mu,\bar{L}}(\sigma,\theta)$.
Using \prettyref{lem:acg_tech_props}(a), we first have 
\begin{equation}
A_{J+1}\geq\frac{1}{L_{0}}\prod_{i=1}^{J}\left(1+\sqrt{\frac{\mu}{2L_{i}}}\right)\geq\frac{1}{\bar{L}}\left(1+\sqrt{\frac{\mu}{2\bar{L}}}\right)^{J}\label{eq:Aj_tech_bd}
\end{equation}
Using the above bound, the fact that $J\geq2\sqrt{2\bar{L}/\mu}\log(\bar{L}{\cal A})$
from the definition in \eqref{eq:acg_complexity}, the bound $\mu\leq\bar{L}$,
and the fact that $\log(1+t)\geq t/2$ on $t\in[0,1]$, it holds that
\[
\log(\bar{L}{\cal A})\leq\frac{J}{2}\sqrt{\frac{\mu}{2\bar{L}}}\leq J\log\left(1+\sqrt{\frac{\mu}{2\bar{L}}}\right)\overset{\eqref{eq:Aj_tech_bd}}{\leq}\log(\bar{L}A_{J+1})
\]
which implies $A_{J+1}\geq{\cal A}$. Hence, it follows from \prettyref{lem:acg_tech_props}(d)
that \eqref{eq:acg_termination} holds. In view of Step~3 of \prettyref{alg:pf_acg}
this implies that termination has to have occurred at or before iteration
$J+1$, which contradicts our initial assumption. Thus, \prettyref{alg:pf_acg}
must have terminated by iteration $J+1$.

(b) This follows immediately from part (a) and \prettyref{lem:acg_tech_props}(c).
\end{proof}
The last result of this subsection shows how to invoke \prettyref{alg:pf_acg}
so that its successful termination implements a single iteration of
\prettyref{alg:gpds}.
\begin{lemma}
\label{lem:acg_specialization}Suppose \prettyref{alg:pf_acg} is
called with $(\psi^{s},\psi^{n})$ as in \eqref{eq:psi_sn_vals} for
some $m>0$ and $\hat{z}\in\dom\psi^{n}$, $\sigma=\rho$, and $y_{0}=\hat{z}$.
If the call terminates in Step~3 with an output triple $(y_{j+1},r_{j+1},L_{j+1})$,
then the quadruple $(z_{k+1},u_{k+1},m_{k+1},z_{k})=(y_{j+1},2mr_{j+1},m,\hat{z})$
satisfies \eqref{eq:gd_incl}--\eqref{eq:gd_ineq2}.
\end{lemma}

\begin{proof}
Using \prettyref{lem:acg_tech_props}(b), it holds that 
\begin{align*}
u_{k+1}=2mr_{j+1} & \in2m\left[\nabla\psi^{s}(y_{j+1})+\pt\psi^{n}(y_{j+1})\right]\\
 & =\nabla f(z_{k+1})+2m_{k+1}(z_{k+1}-z_{k})+\pt\psi^{n}(z_{k+1}).
\end{align*}
which is exactly \eqref{eq:gd_incl}. Now, using the first inequality
in \eqref{eq:acg_termination}, the choice of $\sigma=1/(2\alpha)$,
and the fact that $y_{0}=\hat{z}=z_{k}$, we have
\[
\|u_{k+1}\|^{2}=4m^{2}\|r_{j+1}\|^{2}\overset{\eqref{eq:acg_termination}}{\leq}2\left(\rho m\right)^{2}\|y_{j+1}-y_{0}\|^{2}=2\left(\rho m_{k+1}\right)^{2}\|z_{k+1}-z_{k}\|^{2},
\]
which is exactly \eqref{eq:gd_ineq2}. Finally, the second condition
of \eqref{eq:acg_termination}, the relation $\psi(\cdot)=\phi(\cdot)/(2m_{k+1})+\|\cdot-y_{0}\|^{2}/2$,
and the fact that $y_{0}=\hat{z}=z_{k}$ imply 
\begin{align*}
 & \|u_{k+1}+2m_{k+1}(z_{k}-z_{k+1})\|^{2}=4m_{k+1}^{2}\|r_{j+1}+y_{j+1}-y_{0}\|^{2}\\
 & \overset{\eqref{eq:acg_termination}}{\leq}4\theta m_{k+1}^{2}\left[\psi(y_{0})-\psi(y_{j+1})+\frac{1}{2}\|y_{j+1}-y_{0}\|^{2}\right]=2\theta m_{k+1}\left[\phi(z_{k})-\phi(z_{k+1})\right],
\end{align*}
which is exactly \eqref{eq:gd_ineq1}. Combing all previous inequalities
yields the desired conclusion.
\end{proof}
Some remarks are in order. We first remark on \prettyref{alg:pf_acg_line_search}:
\begin{enumerate}[left=1em]
\item In view of \eqref{eq:descent_acg}, the number of iterations in its
$(j+1)$-th call stops is bounded above by $1+\log_{\beta}(L_{j+1}/L_{j})$.
\item The update for $y$ is equivalent to 
\[
y=\argmin_{z\in\dom\psi^{n}}\left\{ \frac{\psi^{n}(z)}{L+\mu}+\frac{1}{2}\left\Vert z-\left(\tilde{x}-\frac{\nabla\psi^{s}(\tilde{x})}{L+\mu}\right)\right\Vert ^{2}\right\} 
\]
which is a single call to the prox-oracle of $\psi^{n}/(L+\mu)$.
\item The descent condition \eqref{eq:acg_curv_cond} is well-known in existing
literature for adaptive FISTA-type methods (see, for example, \cite[Subsection 4.3]{parikh2014proximal}). 
\end{enumerate}
We now remark on \prettyref{alg:pf_acg} and its associated results:
\begin{enumerate}[left=1em, resume]
\item It is shown in \prettyref{lem:acg_tech_props} that (i) $r_{j+1}$
is a stationarity residual for the iterate $y_{j+1}$ and (ii) $\{L_{j}\}_{j\geq0}$
forms a nondecreasing sequence of nonnegative scalars.
\item Step~1 is generally where most of the computation is done, wherein
(possibly) multiple accelerated gradient steps are performed using
\prettyref{alg:pf_acg_line_search}. It is also the only step that
requires evaluating the prox oracle for $\psi^{n}$.
\item It is shown in \prettyref{prop:acg_alg_props}(b) that both inequalities
in Step~2 hold when $\psi^{s}$ is $\mu$-strongly convex. The first
(resp. second) inequality of \eqref{eq:acg_cvx_cond} is used to ensure
that the first (resp. second) inequality of \eqref{eq:acg_termination}
holds when enough iterations are performed. See the analysis in \prettyref{subsec:prf_acg_tech_props}
for more details.
\item Condition \eqref{eq:acg_termination} is chosen so that \prettyref{alg:pf_acg}
implements a single step of \prettyref{alg:gpds} if it stops in Step~3
and it is given the right inputs (see \prettyref{lem:acg_specialization}).
\item Suppose \prettyref{alg:pf_acg} terminates in $J$ iterations. Then,
the number iterations of \prettyref{alg:pf_acg_line_search} taken
by \prettyref{alg:pf_acg} is
\[
\sum_{j=0}^{J-1}\left[1+\log_{\beta}\frac{L_{j+1}}{L_{j}}\right]=J+\log_{\beta}\frac{L_{J}}{L_{0}}\leq J+\log_{\beta}\frac{\overline{L}}{L_{0}}.
\]
Thus, on average (up to a $(1/J)\log_{\beta}(\overline{L}/L_{0})$
additive term) \prettyref{alg:pf_acg} uses only one accelerated gradient
step or two function and prox oracle calls. It is worth mentioning
that Nesterov's universal fast gradient method \cite[Section 4]{nesterov2015universal}
uses on average (up to a $(1/J)\log_{\beta}(\overline{L}/L_{0})$
additive term) four function/prox oracle calls per invocation.
\end{enumerate}

\subsection{PF.APD Algorithm}

\label{subsec:apd}

Broadly speaking, PF.APD is a \emph{double-loop} method consisting
of \emph{outer iterations }and (possibly) several \emph{inner iterations
}per outer iteration. More specifically, the $(k+1)$-th outer iteration
of PF.APD repeatedly applies \prettyref{alg:pf_acg} to the proximal
subproblem
\[
z_{k+1}\approx\argmin_{z\in\dom h}\left\{ \frac{\phi(z)}{2\hat{m}}+\frac{1}{2}\|z-z_{k}\|^{2}\right\} ,
\]
for increasing values of $\hat{m}>0$, where $z_{k}$ is an approximate
solution to the $k$-th subproblem. On the other hand, the inner iterations
refer to the iterations performed by \prettyref{alg:pf_acg}.

We now present the details of PF.APD and its key properties. To help
our discussion, we first give the complete pseudocode of PF.APD through
\prettyref{alg:pf_acg_line_search} and \prettyref{alg:pf_apd}. More
specifically, \prettyref{alg:pf_acg_line_search} presents the (lower
curvature) line search strategy used in PF.APD, while \prettyref{alg:pf_apd}
describes the other steps of PF.APD and how \prettyref{alg:pf_apd_line_search}
is invoked. 

\begin{algorithm}[h]
\caption{\small Line Search and Proximal Descent Step}
\label{alg:pf_apd_line_search}
\begin{algorithmic}[1]

\StateO \texttt{\small{}Data}{\small{}: $(\psi^{s},\psi^{n},f,h)$
as in }\eqref{eq:psi_sn_vals}{\small{}, $\hat{z}\in\dom h$, $\hat{m}\in\r_{++}$,
$\hat{M}\in[m,\infty)$;}{\small\par}

\StateO \texttt{\small{}Hyper-parameters:}{\small{} $\rho\in(0,1)$,
$\theta\in(2,\infty)$, $\alpha\in(1,\infty)$, $\beta\in(1,\infty)$;}{\small\par}

\StateO \texttt{\small{}Outputs}{\small{}: $(z,u,m,M)\in\dom h\times\rn$;}{\small\par}

\State $M\gets\hat{M}$

\State $\phi(\cdot)\gets f(\cdot)+h(\cdot)$

\For{$\ell\gets0,1,\ldots$}

\State $m\gets\hat{m}\alpha^{\ell}$

\texttt{\textcolor{teal}{\small{}\hspace{-0.5em}$\triangleright$
Step 1:$\ $$(\ell+1)^{{\rm th}}$ proximal subproblem.}}{\small\par}

\State \textbf{call} \prettyref{alg:pf_acg} with data $(\psi^{s},\psi^{n})$,
$y_{0}\equiv\hat{z}$, {\small{}$\mu\equiv1/2$,}{\small\par}

\StateO \quad{}{\small{}$L_{0}\equiv M/(2m)+1$}, and hyper-parameters
$\sigma\equiv\rho$, {\small{}$\theta$, $\beta$, to obtain an}{\small\par}

\StateO \quad{}{\small{}output tuple $(z,r,L)$}{\small\par}

\State $u\gets2mr$

\State $M\gets2m(L-1)$ 

\texttt{\textcolor{teal}{\small{}\hspace{-0.5em}$\triangleright$
Step 2:$\ $Proximal descent check. }}{\small\par}

\State \textbf{if} the inequalities\vspace{-1em}
\begin{equation}
\begin{aligned}\|u+2m(z-\hat{z})\|^{2} & \leq2\theta m\left[\phi(\hat{z})-\phi(z)\right],\\
\|u\|^{2} & \leq2\left(\rho m\right)^{2}\|z-\hat{z}\|^{2},
\end{aligned}
\label{eq:gd_spec_ineq}
\end{equation}
\vspace{-1em}

hold, then \textbf{return} $(z,u,m,M)$

\EndFor 

\end{algorithmic}
\end{algorithm}

\begin{algorithm}[h]
\caption{\small Parameter-Free Accelerated Proximal Descent (PF.APD) Algorithm}
\label{alg:pf_apd}
\begin{algorithmic}[1]

\StateO \texttt{\small{}Data}{\small{}: $(f,h)$ as in \ref{asm:a1}--\ref{asm:a3},
$z_{0}\in\dom h$, $m_{0}\in\r_{++}$, $M_{0}\in[m_{0},\infty)$,
$\varepsilon\in\r_{++}$;}{\small\par}

\StateO \texttt{\small{}Hyper-parameters:}{\small{} $\rho\in(0,1)$,
$\theta\in(2,\infty)$, $\alpha\in(1,\infty)$, $\beta\in(1,\infty)$;}{\small\par}

\StateO \texttt{\small{}Outputs}{\small{}: $(z_{k+1},v_{k+1})\in\dom h\times\rn$;}{\small\par}

\For{$k\gets0,1,\ldots$}

\texttt{\textcolor{teal}{\small{}\hspace{-0.5em}$\triangleright$
Step 1:$\ $Line search for $m_{k+1}$ and proximal descent step.}}{\small\par}

\State $\hat{m}\gets\begin{cases}
m_{k}/\alpha, & \text{if }k\geq1\text{ and }m_{k}<\cdots<m_{0},\\
m_{k}, & \text{otherwise}
\end{cases}$

\State \textbf{call} \prettyref{alg:pf_apd_line_search} with data
\vspace{-0.75em}
\begin{equation}
\psi^{s}(\cdot)=\frac{f(\cdot)}{2\hat{m}}+\frac{1}{2}\|\cdot-z_{k}\|^{2},\quad\psi^{n}(\cdot)=\frac{h(\cdot)}{2\hat{m}},\label{eq:psi_apd_def}
\end{equation}
\vspace{-1.25em}

\StateO \quad{}$(f,h)$, $\hat{z}\equiv z_{k}$, $\hat{m}\equiv\hat{m}$,
$\hat{M}\equiv M_{k}$, and hyper-parameters $\rho$, $\theta$, $\alpha$,
$\beta$ 

\StateO \quad{}to obtain $(z_{k+1},u_{k+1},m_{k+1},M_{k+1})$

\texttt{\textcolor{teal}{\small{}\hspace{-0.5em}$\triangleright$
Step 2:$\ $Stationarity termination check.}}{\small\par}

\State $v_{k+1}\gets2m_{k+1}(u_{k+1}+z_{k}-z_{k+1})$

\If{$\|v_{k+1}\|\leq\varepsilon$}

\State \textbf{return }$(z_{k+1},\:v_{k+1})$

\EndIf

\EndFor 

\end{algorithmic}
\end{algorithm}

We next present three important properties about \prettyref{alg:pf_apd}
and its iterates. As its proof is mostly technical, we move it to
\prettyref{subsec:prf_apd_tech_props}. Moreover, to ensure that the
resulting properties account for the possible asymmetry in \eqref{eq:intro_curv},
we make use of the scalars
\begin{equation}
\begin{aligned}m_{*} & :=\argmin_{z,z'\in\dom h,\ t\geq0}\left\{ t:f(z)-\ell_{f}(z;z')\geq-\frac{t}{2}\|z-z\|^{2}\right\} ,\\
M_{*} & :=\argmin_{z,z'\in\dom h,\ t\geq0}\left\{ t:f(z)-\ell_{f}(z;z')\leq\frac{t}{2}\|z-z\|^{2}\right\} ,
\end{aligned}
\label{eq:mM_star_defs}
\end{equation}
which are the values of a curvature pair of $f$.
\begin{proposition}
\label{prop:apd_tech_props}Define the scalars
\begin{equation}
\begin{gathered}\overline{m}:=\max\{m_{0},(\alpha+\beta)m_{*}\},\quad\overline{M}:=\beta\left[\max\{M_{0},M_{*}\}+2\overline{m}\right],\\
\overline{{\cal L}}_{0}:=\frac{\overline{M}}{2m_{0}}+1,\quad P_{0}:=\log^{1+}\left\{ \overline{{\cal L}}_{0}{\cal A}_{\frac{1}{2},\overline{{\cal L}}_{0}}\left(\rho,\theta\right)\right\} ,
\end{gathered}
\label{eq:apd_scalars}
\end{equation}
where $(m_{*},M_{*})$ and ${\cal A}_{\mu,\bar{L}}(\cdot,\cdot)$
are as in \eqref{eq:mM_star_defs} and \eqref{eq:calA_def}, respectively.
Then, for every $k\geq0$, the following statements hold about \prettyref{alg:pf_apd}
and its iterates:
\end{proposition}

\begin{itemize}
\item[(a)] $M_{k}\leq M_{k+1}\leq\overline{M}<\infty$ and $\{1/m_{k}\}$ is
bitonic\footnote{A sequence $\{a_{k}\}_{k=0}^{n}$ is \emph{bitonic} if there exists
$0\leq j\leq n$ such that $a_{0}\leq\cdots\leq a_{j}\geq\cdots\geq a_{n}$.
Note that monotone sequences are bitonic as well.} and bounded below by $1/\overline{m}$;
\item[(b)] its $(k+1)$-th outer iteration performs at most $T_{k+1}$ inner
iterations, where 
\begin{equation}
T_{k+1}\leq20\left(1+\log_{\alpha}\frac{m_{k+1}}{m_{k}}+\frac{1}{\sqrt{\alpha}-1}\sqrt{\frac{\overline{M}}{2m_{k}}}\right)P_{0};\label{eq:Tk_def}
\end{equation}
\item[(c)] it performs a finite number of outer iterations $K(\varepsilon)$,
where
\begin{equation}
K(\varepsilon)\leq1+\sum_{k=0}^{K(\varepsilon)-2}\frac{\overline{m}}{m_{k+1}}<1+\frac{2\theta\Delta_{0}\overline{m}}{\varepsilon^{2}};\label{eq:K_rho_gen_bds}
\end{equation}
\item[(d)] if, in addition, $f$ is convex and $\rho^{2}\alpha<1$, then $m_{k}=\alpha^{-k}m_{0}$
for every $k\geq0$ and $K(\varepsilon)$ in \eqref{eq:K_rho_gen_bds}
also satisfies 
\begin{equation}
K(\varepsilon)\leq1+\log_{\alpha}\left[1+\frac{4\theta\alpha^{-1}m_{0}^{2}\cdot R_{2m_{0}/\alpha}(z_{0})}{\varepsilon^{2}}\right],\label{eq:K_rho_spec_bds}
\end{equation}
where $R_{\tau}(\cdot)$ is as in \eqref{eq:R_phi_def};
\item[(e)] $v_{k+1}\in\nabla f(z_{k+1})+\pt h(z_{k+1})$ and its final iterate
$(\bar{z},\bar{v})=(z_{k+1},v_{k+1})$ solves Problem~\textbf{\emph{$\boldsymbol{{\cal CO}}$}}\emph{.}
\end{itemize}
We are now ready to give some important iteration complexity bounds
on \prettyref{alg:pf_apd}.
\begin{theorem}
\label{thm:apd_main_compl}Define $Q_{0}:=20P_{0}\cdot\left[1+\log_{\alpha}(\overline{m}/m_{0})\right],$where
$\overline{m}$ and $P_{0}$ are as in \eqref{eq:apd_scalars}, respectively.
Then, \prettyref{alg:pf_apd} stops and outputs a pair $(\bar{z},\bar{v})=(z_{k+1},v_{k+1})$
solving Problem~\textbf{\emph{$\boldsymbol{{\cal CO}}$}} in \textbf{$\overline{T}$}
inner iterations, where
\begin{equation}
\overline{T}\leq Q_{0}+P_{0}\left(\frac{20}{\sqrt{\alpha}-1}\right)\sqrt{\overline{M}\left[1+\frac{2\theta\Delta_{0}\overline{m}}{\varepsilon^{2}}\right]\left[\frac{1}{m_{0}}+\frac{2\theta\Delta_{0}}{\varepsilon^{2}}\right]},\label{eq:apd_gen_compl}
\end{equation}
and $\Delta_{0}$ is as in \eqref{eq:optimality_residuals}. Moreover,
if $f$ is convex and $\rho^{2}\alpha<1$, then
\begin{equation}
\overline{T}\leq Q_{0}+P_{0}\left[\frac{20\alpha}{(\sqrt{\alpha}-1)^{2}}\right]\sqrt{\overline{M}\left[\frac{1}{m_{0}}+\frac{\theta\min\{\Delta_{0},m_{0}d_{0}^{2}/\alpha\}}{\varepsilon^{2}}\right]},\label{eq:apd_spec_compl}
\end{equation}
where $d_{0}$ is as in \eqref{eq:optimality_residuals}.
\end{theorem}

\begin{proof}
The fact that \prettyref{alg:pf_apd} stops in a finite number of
inner iterations with a pair solving Problem~${\cal CO}$ is immediate
from \prettyref{prop:apd_tech_props}. Furthermore, the previous proposition
also implies that the total number of inner iterations in a single
call of \prettyref{alg:pf_apd} is at most 
\begin{align}
\sum_{k=0}^{K(\varepsilon)-1}T_{k+1} & \leq20P_{0}\sum_{k=0}^{K(\varepsilon)-1}\left(1+\log_{\alpha}\frac{m_{k+1}}{m_{k}}+\frac{1}{\sqrt{\alpha}-1}\sqrt{\frac{\overline{M}}{2m_{k}}}\right)\nonumber \\
 & \leq20P_{0}\left(1+\log_{\alpha}\frac{m_{K(\varepsilon)+1}}{m_{0}}+\frac{\sqrt{\overline{M}}}{\sqrt{\alpha}-1}\sum_{k=0}^{K(\varepsilon)-1}\frac{1}{\sqrt{m_{k}}}\right)\nonumber \\
 & \leq Q_{0}+\frac{20P_{0}\sqrt{\overline{M}}}{\sqrt{\alpha}-1}\sum_{k=0}^{K(\varepsilon)-1}\frac{1}{\sqrt{m_{k}}},\label{eq:Tk_bd}
\end{align}
where $T_{k+1}$ and $K(\varepsilon)$ are as in \eqref{eq:Tk_def}
and \eqref{eq:K_rho_gen_bds}, respectively. Let us now bound the
sum $\sum_{k=0}^{K(\varepsilon)-1}m_{k}^{-1/2}$. Using \prettyref{prop:apd_tech_props}(c)
and the fact that $\|z\|_{1}\leq\sqrt{n}\|z\|_{2}$ for any $z\in\rn$,
we first have 
\[
\sum_{k=0}^{K(\varepsilon)-1}\frac{1}{\sqrt{m_{k}}}\leq\left[K(\varepsilon)\sum_{k=0}^{K(\varepsilon)-1}\frac{1}{m_{k}}\right]^{1/2}\leq\sqrt{\left(1+\frac{2\theta\Delta_{0}\overline{m}}{\varepsilon^{2}}\right)\left(\frac{1}{m_{0}}+\frac{2\theta\Delta_{0}}{\varepsilon^{2}}\right)}.
\]
Using \eqref{eq:Tk_bd} and the above bound yields \eqref{eq:apd_gen_compl}. 

Now, let ${\cal R}_{0}:=R_{2m_{0}/\alpha}(z_{0})$ and suppose $f$
is convex. Using \prettyref{prop:apd_tech_props}(d), \eqref{eq:R_bd}
with $m_{1}=m_{0}/\alpha$, and the inequality $\sqrt{a+b}\leq\sqrt{a}+\sqrt{b}$
for $a,b\in\r$, we have

\begin{align*}
\sum_{k=0}^{K(\varepsilon)-1}\frac{1}{\sqrt{m_{k}}} & =\sum_{k=0}^{K(\varepsilon)-1}\sqrt{\frac{\alpha^{k}}{m_{0}}}\leq\frac{\alpha^{K(\varepsilon)/2}}{\sqrt{m_{0}}(\sqrt{\alpha}-1)}\overset{(d)}{\leq}\frac{\alpha}{\sqrt{m_{0}}(\sqrt{\alpha}-1)}\sqrt{1+\frac{4\theta\alpha^{-1}m_{0}^{2}{\cal R}_{0}}{\varepsilon^{2}}}\\
 & \leq\frac{\alpha}{\sqrt{m_{0}}(\sqrt{\alpha}-1)}\sqrt{1+\frac{\theta m_{0}\min\{\Delta_{0},m_{0}d_{0}^{2}/\alpha\}}{\varepsilon^{2}}}\\
 & =\frac{\alpha}{\sqrt{\alpha}-1}\sqrt{\frac{1}{m_{0}}+\frac{\theta\min\{\Delta_{0},m_{0}d_{0}^{2}/\alpha\}}{\varepsilon^{2}}.}
\end{align*}
Combining \eqref{eq:Tk_bd} and the above bound yields \eqref{eq:apd_spec_compl}. 
\end{proof}
Some remarks are in order. We first remark on \prettyref{alg:pf_apd_line_search}:
\begin{enumerate}[left=1em]
\item In view of assumption \ref{asm:a2} and \prettyref{prop:acg_alg_props},
the number of iterations in its $k$-th call is bounded above by $1+\log_{\alpha}(m_{k+1}/m_{k})$.
\item The checks in its Step~2 correspond to \eqref{eq:gd_ineq1} and \eqref{eq:gd_ineq2},
respectively.
\item If the $\ell$-th call to \prettyref{alg:pf_acg} ends with a ``bad
termination'', i.e., Step~2 in \prettyref{alg:pf_acg}, then \eqref{eq:gd_spec_ineq}
does not hold, the estimate $m$ is increased by a factor of $\alpha$,
and the algorithm proceeds to the $(\ell+1)$-th iteration.
\end{enumerate}
We now remark on \prettyref{alg:pf_apd} and its associated results:
\begin{enumerate}[left=1em, resume]
\item It is shown in \prettyref{prop:apd_tech_props} that (i) $v_{j+1}$
is a stationarity residual for the iterate $z_{j+1}$ and (ii) $\{M_{k}\}_{k\geq0}$
and $\{m_{k}\}_{k\geq0}$ are nondecreasing and nonnegative.
\item $Q_{0}$ in \eqref{eq:apd_gen_compl}--\eqref{eq:apd_spec_compl}
bounds the total number of inner iterations performed by unsuccessful
calls to \prettyref{alg:pf_acg}, i.e., those that stop in Step~2
of \prettyref{alg:pf_acg}.
\item While $m_{0}$ and $M_{0}$ are free parameters, a good initial value\footnote{This is motivated by the fact that $m_{0}$ and $M_{0}$ are bounded
by the Lipschitz constant of $\nabla f$.} for them is an estimate of the local Lipschitz constant $\tilde{L}_{0}$
of $\nabla f$ at $z_{0}$. Similar to the approach in \cite{nesterov1983},
one can estimate $\tilde{L}_{0}$ by sampling some $\hat{z}\in\dom h$
with $\hat{z}\neq z_{0}$ and choosing $\tilde{L}_{0}=\|\nabla f(z_{0})-\nabla f(\hat{z})\|/\|z_{0}-\hat{z}\|$.
\end{enumerate}
Before ending the section, we discuss how different choices of $m_{0}$
affect the complexities in \eqref{eq:apd_gen_compl} and \eqref{eq:apd_spec_compl}
when $m_{*}\leq M_{*}$:
\begin{enumerate}[left=1em, resume]
\item In the general case, choosing $m_{0}=1$ implies that the bound in
\eqref{eq:apd_gen_compl} (resp. \eqref{eq:apd_spec_compl}) is ${\cal O}(\sqrt{M_{*}m_{*}}\Delta_{0}/\varepsilon^{2})$
(resp. ${\cal O}(\sqrt{M_{*}\Delta_{0}}/\varepsilon)$) which matches
the complexity of the AIPP in \cite{kong2019complexity} and is optimal\footnote{See \cite[Theorem 4.7]{zhou2019lower}.}
for finding stationary points of \eqref{sec:intro} in the weakly-convex
(resp. convex) setting in terms of of $m_{*}$, $M_{*}$, $\Delta_{0}$,
and $\varepsilon$. 
\item If $d_{0}$ is known, then choosing $m_{0}=\varepsilon/d_{0}$ implies
\eqref{eq:apd_spec_compl} is ${\cal O}(\sqrt{M_{*}d_{0}}\log\varepsilon^{-1}/\sqrt{\varepsilon})$
which is optimal\footnote{See \cite[Section 2.2.2]{nesterov2018lectures} or \cite[Theorem 1]{carmon2021lower}.},
up to logarithmic terms, for finding stationary points of \eqref{sec:intro}
in the convex setting in terms of of $M_{*}$, $d_{0}$, and $\varepsilon$.
\end{enumerate}

\section{Technical Proofs}

\label{sec:tech_proofs}

This section gives the proofs of several technical results in \prettyref{sec:pf_algs}.
More specifically, it presents the proofs of \prettyref{lem:acg_tech_props}
and \prettyref{prop:apd_tech_props}.

\subsection{Proof of \prettyref{lem:acg_tech_props}}

\label{subsec:prf_acg_tech_props}

To avoid repetition, we let 
\begin{equation}
\{(A_{j},\tilde{x}_{j},y_{j},x_{j},L_{j})\}_{j\geq0}\label{eq:acg_iters}
\end{equation}
denote the sequence of iterates generated by a single call to \prettyref{alg:pf_acg}
and define
\[
\begin{gathered}a_{i}:=A_{i+1}-A_{i},\quad\xi_{i}:=1+\mu A_{i},\\
\tilde{q}_{i+1}(\cdot):=\ell_{\psi^{s}}(\cdot;\tilde{x}_{i})+\psi^{n}(\cdot)+\frac{\mu}{2}\|\cdot-\tilde{x}_{i}\|^{2},\\
q_{i+1}(\cdot):=\tilde{q}_{i+1}(y_{i+1})+L_{i+1}\left\langle \tilde{x}_{i}-y_{i+1},\cdot-y_{i+1}\right\rangle +\frac{\mu}{2}\|\cdot-y_{i+1}\|^{2},
\end{gathered}
\]
for every $i\geq0$. Recall also that each iterate in \eqref{eq:acg_iters}
is obtained in a finite number of iterations of \prettyref{alg:pf_acg_line_search}
in view of \eqref{eq:descent_acg} and \eqref{eq:acg_curv_cond}.

We first present some basic technical properties about $\tilde{q}$
and $q$.
\begin{lemma}
\label{lem:acg_q_props}If $\psi^{s}$ is $\mu$-strongly convex,
then, for every $j\geq0$,
\begin{itemize}
\item[(a)] $\tilde{q}_{j+1}(y_{j+1})=q_{j+1}(y_{j+1})$ and $\tilde{q}_{j+1}(\cdot)\leq q_{j+1}(\cdot)\leq\psi(\cdot)$;
\item[(b)] $y_{j+1}=\min_{x\in\rn}\left\{ q_{j+1}(x)+L_{j+1}\|x-\tilde{x}_{j+1}\|^{2}/2\right\} $;
\item[(c)] $x_{j+1}=\argmin_{x\in\rn}\left\{ a_{j}q_{j+1}(x)+\xi_{j+1}\|x-x_{j}\|^{2}/2\right\} $.
\end{itemize}
\end{lemma}

\begin{proof}
(a) See \cite[Lemma B.0.1]{kong2021accelerated}.

(b) Let $\Psi(\cdot)=q_{j+1}(\cdot)+L_{j+1}\|\cdot-\tilde{x}_{j+1}\|^{2}/2$.
It follows from the definition of $q_{j+1}$ that $\nabla\Psi(y_{j+1})=0$
and, hence, $y_{j+1}$ satisfies the optimality condition of the given
inclusion.

(c) Using the definition of $q_{j+1}$, the given optimality condition
of $x_{j+1}$ holds if and only if
\[
x_{j+1}=x_{j}-\frac{a_{j}\nabla q_{j+1}(x_{j})}{\xi_{j+1}}=x_{j}+\frac{a_{j}\left[L(y_{j+1}-\tilde{x}_{j})+\mu(y_{j+1}-x_{j})\right]}{1+\mu A_{j+1}}
\]
which is equivalent to the update for $x_{j+1}$ in \prettyref{alg:pf_acg}
(given by \prettyref{alg:pf_acg_line_search}).
\end{proof}
The next result presents an important technical bound on the residual
$\|y_{j+1}-\tilde{x}_{j}\|^{2}$.
\begin{lemma}
\label{lem:main_acg_bd}If $\psi^{s}$ is $\mu$-strongly convex,
then, for every $j\geq0$ and $y\in\rn$, 
\begin{align}
 & \frac{\mu A_{j+1}}{2}\|y_{j+1}-\tilde{x}_{j}\|^{2}+A_{j+1}\psi(y_{j+1})+\frac{\xi_{j+1}}{2}\|y-x_{j+1}\|^{2}\label{eq:main_acg_bd}\\
 & \leq A_{j}q_{j+1}(y_{j})+a_{j}q_{j+1}(y)+\frac{\xi_{j}}{2}\|y-x_{j}\|^{2}.\nonumber 
\end{align}
\end{lemma}

\begin{proof}
Let $y\in\rn$ be fixed. We first derive two auxiliary technical inequalities.
For the first one, we use the fact that $a_{j}q_{j+1}+\xi_{j}\|\cdot-x_{j}\|^{2}/2$
is $\xi_{j+1}$-strongly convex, the definition of $\xi_{j+1}$, and
the optimality of $x_{j+1}$ in \prettyref{lem:acg_q_props}(c) to
obtain 
\begin{equation}
a_{j}q_{j+1}(y)+\frac{\xi_{j}}{2}\|y-x_{j}\|^{2}-\frac{\xi_{j+1}}{2}\|y-x_{j+1}\|^{2}\geq a_{j}q_{j+1}(x_{j+1})+\frac{\xi_{j}}{2}\|x_{j+1}-x_{j}\|^{2}.\label{eq:xjp1_opt1}
\end{equation}
For the second one, let $r_{j+1}:=(A_{j}y_{j}+a_{j}x_{j+1})/A_{j+1}$.
Using the convexity of $q_{j+1}$, the updates in \prettyref{alg:pf_acg_line_search}
and \prettyref{alg:pf_acg}, and \prettyref{lem:acg_q_props}(a)--(b),
we obtain
\begin{align}
 & A_{j}q_{j+1}(y_{j})+a_{j}q_{j+1}(x_{j+1})+\frac{\xi_{j}}{2}\|x_{j+1}-x_{j}\|^{2}\nonumber \\
 & \geq A_{j+1}\left[q_{j+1}(r_{j+1})+\frac{\xi_{j}}{2a_{j}^{2}}\left\Vert r_{j+1}-\frac{A_{j}y_{j}+a_{j}x_{j}}{A_{j+1}}\right\Vert ^{2}\right]\nonumber \\
 & =A_{j+1}\left[q_{j+1}(r_{j+1})+\frac{L_{j+1}}{2}\left\Vert r_{j+1}-\tilde{x}_{j}\right\Vert ^{2}\right]\geq A_{j+1}\min_{x\in\rn}\left\{ q_{j+1}(x)+\frac{L_{j+1}}{2}\|x-\tilde{x}_{j}\|^{2}\right\} \nonumber \\
 & \overset{\text{\prettyref{lem:acg_q_props}(a)-(b)}}{=}A_{j+1}\left[\tilde{q}_{j+1}(y_{j+1})+\frac{L_{j+1}}{2}\|y_{j+1}-\tilde{x}_{j}\|^{2}\right]\label{eq:xjp1_opt2}
\end{align}
Combining \eqref{eq:xjp1_opt1}, \eqref{eq:xjp1_opt2}, and \eqref{eq:acg_curv_cond}
with $L=L_{j+1}$, we conclude that
\begin{align*}
 & A_{j}q_{j+1}(y_{j})+a_{j}q_{j+1}(y)+\frac{\xi_{j}}{2}\|y-x_{j}\|^{2}-\frac{\xi_{j+1}}{2}\|y-x_{j+1}\|^{2}\\
 & \overset{\eqref{eq:xjp1_opt1}}{\geq}A_{j}q_{j+1}(y_{j})+a_{j}q_{j+1}(x_{j+1})+\frac{\xi_{j}}{2}\|x_{j+1}-x_{j}\|^{2}\\
 & \overset{\eqref{eq:xjp1_opt2}}{\geq}\tilde{q}_{j+1}(y_{j+1})+\frac{L_{j+1}}{2}\|y_{j+1}-\tilde{x}_{j}\|^{2}\overset{\eqref{eq:acg_curv_cond}}{\geq}\psi(y_{j+1})+\frac{\mu}{2}\|y_{j+1}-\tilde{x}_{j}\|^{2}.
\end{align*}
\end{proof}
The following result further refines the previous bound on $\|y_{j+1}-\tilde{x}_{j}\|^{2}$.
\begin{lemma}
\label{lem:acg_resid_bd}If $\psi^{s}$ is $\mu$-strongly convex,
then, for every $j\geq0$,
\begin{equation}
\mu A_{j+1}\|y_{j+1}-\tilde{x}_{j}\|^{2}\leq\|y_{j+1}-y_{0}\|^{2}-\xi_{j+1}\|y_{j+1}-x_{j+1}\|^{2}.\label{eq:acg_resid_bd}
\end{equation}
\end{lemma}

\begin{proof}
Let $j\geq0$ be fixed and suppose $\psi^{s}$ is $\mu$-strongly
convex. Moreover, define 
\[
\Psi_{i}:=A_{i}\left[\psi(y_{i})-\psi(y_{j})\right]+\frac{\xi_{i}}{2}\|y_{j}-x_{i}\|^{2}\quad\forall i\geq0.
\]
Using \prettyref{lem:main_acg_bd} with $y=y_{j}$, \prettyref{lem:acg_q_props}(a),
the fact that $a_{j}=A_{j+1}-A_{j}$, and the definition of $\Psi_{i}$
above, we have that for every $i\geq0$,
\begin{align*}
 & \frac{\mu A_{i+1}}{2}\|y_{i+1}-\tilde{x}_{i}\|^{2}\\
 & \overset{\eqref{eq:main_acg_bd}}{\leq}A_{i}q_{i+1}(y_{i})+a_{i}q_{i+1}(y_{j})+\frac{\xi_{i}}{2}\|y_{i}-x_{i}\|^{2}-\Psi_{i+1}-A_{i+1}\psi(y_{j})\\
 & \overset{\prettyref{lem:acg_q_props}\text{(a)}}{\leq}A_{i}\psi(y_{i})+a_{i}\psi(y_{j})+\frac{\xi_{i}}{2}\|y_{i}-x_{i}\|^{2}-\Psi_{i+1}-A_{i+1}\psi(y_{j})\\
 & =\Psi_{i}-\Psi_{i+1}.
\end{align*}
Summing the above inequality from $i=0$ to $j$ and using the fact
that $A_{i+1}\geq0$ for every $i$ and $(x_{0},A_{0},\xi_{0})=(y_{0},0,1)$,
we conclude that 
\begin{align*}
 & \frac{\mu A_{j+1}}{2}\|y_{j+1}-\tilde{x}_{j}\|^{2}\leq\sum_{i=0}^{j}\frac{\mu A_{i+1}}{2}\|y_{i+1}-\tilde{x}_{i}\|^{2}\leq\Psi_{0}-\Psi_{j+1}\\
 & =\frac{\xi_{0}}{2}\|y_{j}-x_{0}\|^{2}-\frac{\xi_{j}}{2}\|y_{j+1}-x_{j+1}\|^{2}=\frac{1}{2}\|y_{j}-y_{0}\|^{2}-\frac{\xi_{j}}{2}\|y_{j+1}-x_{j+1}\|^{2}.
\end{align*}
\end{proof}
We are now ready to prove \prettyref{lem:acg_tech_props}.
\begin{proof}[Proof of \prettyref{lem:acg_tech_props}]
(a) See \cite[Lemma B.0.2]{kong2021accelerated} for the bound on
$A_{j+1}$. The bound on $L_{j}$ follows from how \prettyref{alg:pf_acg_line_search}
is called in \prettyref{alg:pf_acg}, the update rule for $L$ in
\prettyref{alg:pf_acg_line_search}, and \eqref{eq:descent_acg} which
follows from assumption \ref{asm:b2}.

(b) Using the optimality of $y_{j+1}$ given by \prettyref{alg:pf_acg_line_search}
and \prettyref{alg:pf_acg} and the definition of $r_{j+1}$, it follows
that 
\[
0\in\nabla\psi^{s}(\tilde{x}_{j})+\partial\psi^{n}(y_{j+1})+(L_{j+1}+\mu)(y_{j+1}-\tilde{x}_{j})=\nabla\psi^{s}(y_{j+1})+\partial\psi^{n}(y_{j+1})-r_{j+1}.
\]

(c) The first bound in \eqref{eq:acg_curv_cond} is an immediate consequence
of \prettyref{lem:acg_resid_bd}. For the second bound in \eqref{eq:acg_curv_cond},
note that part (b) and the assumption that $\psi^{s}$ implies that
$r_{j+1}\in\pt\psi(y_{j+1})$. The conclusion now follows from the
previous inclusion and the definition of the subdifferential.

(d) Suppose $A_{j+1}\geq{\cal A}_{\mu,\bar{L}}:={\cal A}_{\mu,\bar{L}}(\sigma,\theta)$
and \eqref{eq:acg_curv_cond} holds. We separate this proof into two
parts. We first prove the bound in \eqref{eq:acg_curv_cond}. Using
the definitions of $r_{j+1}$ and $\bar{L}$, part (c), the fact that
$\mu\leq L_{0}\leq L_{j+1}$, assumption \ref{asm:b2}, and the relation
$(a+b)^{2}\leq2a^{2}+2b^{2}$ for $a,b\in\r$, we have that 
\begin{align*}
\|r_{j+1}\|^{2} & =\|\nabla\psi^{s}(y_{j+1})-\nabla\psi^{s}(\tilde{x}_{j})+(L_{j+1}+\mu)(\tilde{x}_{j}-y_{j+1})\|^{2}\\
 & \leq2\|\nabla\psi^{s}(y_{j+1})-\nabla\psi^{s}(\tilde{x}_{j})\|^{2}+2(L_{j+1}+\mu)^{2}\|\tilde{x}_{j}-y_{j+1}\|^{2}\\
 & \leq2[L_{*}^{2}+(L_{j+1}+\mu)^{2}]\|\tilde{x}_{j}-y_{j+1}\|\leq16\bar{L}^{2}\|\tilde{x}_{j}-y_{j+1}\|^{2}\\
 & \overset{\eqref{eq:acg_resid_bd}}{\leq}\frac{16\bar{L}}{\mu A_{j+1}}\|y_{j+1}-y_{0}\|^{2}.
\end{align*}
It follows from the above bound and the definition of ${\cal A}_{\mu,\bar{L}}$
that 
\[
\|r_{j+1}\|^{2}\leq\frac{16\bar{L}}{\mu A_{j+1}}\|y_{j+1}-y_{0}\|^{2}\leq\frac{16\bar{L}^{2}}{\mu{\cal A}_{\mu,\bar{L}}}\|y_{j+1}-y_{0}\|^{2}\leq\sigma^{2}\|y_{j+1}-y_{0}\|^{2}
\]
and, hence, the first condition of \eqref{eq:acg_termination} holds. 

To show the second condition of \eqref{eq:acg_termination}, let $\gamma:=\sqrt{(2-\theta)/\theta}$.
Using the fact that $\gamma\in(0,1)$, \eqref{eq:acg_curv_cond},
$\mu\leq L_{j+1}$, and the bound 
\[
\|a+b\|^{2}\leq(1+\gamma)\|a\|^{2}+(1+\gamma^{-1})\|b\|^{2}\quad\forall a,b\in\r^{n},
\]
we then have that
\begin{align*}
 & \|r_{j+1}\|^{2}\overset{\eqref{eq:acg_curv_cond}}{\leq}\frac{L^{2}}{\mu A_{j+1}}\|y_{j+1}-y_{0}\|^{2}\leq\frac{4(\mu+L_{j+1})^{2}}{\mu A_{j+1}}\|y_{j+1}-y_{0}\|^{2}\\
 & \leq\frac{16\bar{L}^{2}}{\mu{\cal A}_{\mu,\bar{L}}(\sigma,\theta)}\|y_{j+1}-y_{0}\|^{2}\leq\frac{\gamma^{2}}{4}\|y_{j+1}-y_{0}\|^{2}\overset{\gamma\in(0,1)}{\leq}\left(\frac{\gamma}{1+\gamma}\right)^{2}\|y_{j+1}-y_{0}\|^{2}\\
 & \leq\left(\frac{\gamma}{1+\gamma}\right)^{2}\left(1+\gamma\right)\|r_{j+1}+y_{j+1}-y_{0}\|^{2}+\left(\frac{\gamma}{1+\gamma}\right)^{2}\left(1+\frac{1}{\gamma}\right)\|r_{j+1}\|^{2}\\
 & =\frac{\gamma^{2}}{1+\gamma}\|r_{j+1}+y_{j+1}-y_{0}\|^{2}+\frac{\gamma}{1+\gamma}\|r_{j+1}\|^{2},
\end{align*}
which implies $\|r_{j+1}\|^{2}\leq\gamma^{2}\|r_{j+1}+y_{j+1}-y_{0}\|^{2}$.
It then follows from the second bound in \eqref{eq:acg_curv_cond}
and the previous inequality that 
\begin{align*}
2\left[\psi(y_{0})-\psi(y_{j+1})\right] & \overset{\eqref{eq:acg_cvx_cond}}{\geq}2\left\langle r_{j+1},y_{0}-y_{j+1}\right\rangle \\
 & =\|r_{j+1}+y_{0}-y_{j+1}\|^{2}-\|r_{j+1}\|^{2}-\|y_{0}-y_{j+1}\|^{2}\\
 & \geq(1-\gamma^{2})\|r_{j+1}+y_{0}-y_{j+1}\|^{2}-\|y_{0}-y_{j+1}\|^{2}\\
 & =\frac{2}{\theta}\|r_{j+1}+y_{0}-y_{j+1}\|^{2}-\|y_{0}-y_{j+1}\|^{2}.
\end{align*}
\end{proof}

\subsection{Proof of \prettyref{prop:apd_tech_props}}

\label{subsec:prf_apd_tech_props}
\begin{proof}[Proof of \prettyref{prop:apd_tech_props}]
 (a) Note that the $k$-th successful call of \prettyref{alg:pf_acg}
is such that its input $\psi^{s}$ has the curvature pair 
\begin{equation}
(L_{k+1}^{-},L_{k+1}^{+}):=\left(\max\left\{ 0,\frac{m_{*}}{2m_{k+1}}-1\right\} ,\frac{M_{*}}{2m_{k+1}}+1\right).\label{eq:curv_psis}
\end{equation}
Hence, it follows from Step~1 of \prettyref{alg:pf_apd_line_search},
\prettyref{prop:acg_alg_props}(b) with $\mu=1/2$, and the definition
of $\overline{m}$ imply that the last call of \prettyref{alg:pf_acg}
at the $k$-th iteration of \prettyref{alg:pf_apd} obtains $m_{k+1}$
being at most $\alpha m_{k}\leq\overline{m}$. Consequently, $\{1/m_{k}\}$
(resp. $\{m_{k}\}$) is bounded below by $1/\overline{m}$ (resp.
bounded above by $\overline{m}$). The fact that $\{1/m_{j}\}$ is
bitonic follows from the the definition of $\hat{m}$ in Step 1 of
\prettyref{alg:pf_apd}, the call to \prettyref{alg:pf_apd_line_search}
in of \prettyref{alg:pf_apd}, and the fact that in \prettyref{alg:pf_apd}
the returned scalar $m$ in is always lower bounded by the input $\hat{m}$.
To show the bound on $M_{k}$, note that the curvature pair of $\psi^{s}$
in \eqref{eq:curv_psis} implies that $\nabla\psi^{s}$ is $L_{*}$-Lipschitz
continuous where $L_{*}=\max\{L_{k+1}^{-},L_{k+1}^{+}\}$. It then
follows from the upper previous bound on $m_{k+1}$ and \prettyref{lem:acg_tech_props}(a)
that
\begin{align*}
\frac{M_{k}}{2m_{k+1}}+1 & \leq\frac{M_{k+1}}{2m_{k+1}}+1\leq\beta\left[\frac{\max\left\{ M_{0},M_{*}\right\} }{2m_{k+1}}+1\right]\\
 & \leq\frac{\beta\left[\max\{M_{0},M_{*}\}+2\overline{m}\right]}{2m_{k+1}}=\frac{\overline{M}}{2m_{k+1}},
\end{align*}
which immediately implies $M_{k+1}\geq M_{k}$ and $M_{k+1}\leq\overline{M}$.

(b) Let an outer iteration index $k\geq1$ be fixed and define 
\[
{\cal L}_{\ell}:=\frac{\overline{M}}{2m_{k}\alpha^{\ell}}+1,\quad{\cal I}_{\ell}:=\left\lceil 1+4\sqrt{{\cal L}_{\ell}}P_{0}\right\rceil ,\quad\bar{\ell}:=1+\log_{\alpha}(m_{k+1}/m_{k}),
\]
where $P_{0}$ is as in \eqref{eq:apd_scalars}. Using \prettyref{prop:acg_alg_props}(a)
with $(\mu,\sigma)=(1/2,\rho)$, part (a), the fact that $P_{0}\geq1$,
and assumptions \ref{asm:a1}--\ref{asm:a2}, it follows that the
number of inner iterations performed by \prettyref{alg:pf_apd} at
outer iteration $k$ is bounded above by 
\begin{align*}
\sum_{\ell=0}^{\bar{\ell}}{\cal I}_{\ell} & \leq2\sum_{\ell=0}^{\bar{\ell}}\left(1+4\sqrt{{\cal L}_{\ell}}P_{0}\right)\leq2\sum_{\ell=0}^{\bar{\ell}}\left(1+4\left[\sqrt{\frac{\overline{M}}{2m_{k}\alpha^{\ell}}}+1\right]P_{0}\right)\\
 & \leq10P_{0}\sum_{\ell=0}^{\bar{\ell}}\left(\sqrt{\frac{\overline{M}}{2m_{k}\alpha^{\ell}}}+1\right)=10\left[\bar{\ell}+\sqrt{\frac{\overline{M}}{2m_{k}}}\sum_{\ell=0}^{\bar{\ell}}\alpha^{-\ell/2}\right]P_{0}\\
 & =10\left[\bar{\ell}+\sqrt{\frac{\overline{M}}{2m_{k}}}\left(\frac{1-\alpha^{-\bar{\ell}/2}}{\sqrt{\alpha}-1}\right)\right]P_{0}\leq10\left[\bar{\ell}+\frac{1}{\sqrt{\alpha}-1}\sqrt{\frac{\overline{M}}{2m_{k}}}\right]P_{0}\\
 & \leq20\left[1+\log_{\alpha}\frac{m_{k+1}}{m_{k}}+\frac{1}{\sqrt{\alpha}-1}\sqrt{\frac{\overline{M}}{2m_{k}}}\right]P_{0}.
\end{align*}

(c) In view of \prettyref{prop:apd_tech_props}(a), let $\overline{K}$
be an index satisfying 
\[
\frac{\overline{K}-1}{\overline{m}}\leq\sum_{k=0}^{\overline{K}-2}\frac{1}{m_{k+1}}<\frac{2\theta\Delta_{0}}{\varepsilon^{2}}\leq\sum_{k=0}^{\overline{K}-1}\frac{1}{m_{k+1}}.
\]
Using \prettyref{lem:acg_specialization}, the choice of inputs to
\prettyref{alg:pf_acg}, \prettyref{lem:basic_gd_compl}(b), and the
last of the above inequalities, we have that
\begin{align*}
\inf_{0\leq k\leq\overline{K}-1}\|v_{j+1}\|^{2} & \leq\frac{2\theta\Delta_{0}}{\sum_{k=0}^{\overline{K}-1}m_{k+1}^{-1}}\leq\varepsilon^{2}.
\end{align*}
Hence, because of the termination condition in Step 2 of \prettyref{alg:pf_apd},
it follows that the number of outer iterations $K(\varepsilon)$ is
at most $\overline{K}$. Using the fact that $m_{k+1}>0$ for every
$k\geq0$, the bounds in \eqref{eq:K_rho_gen_bds} immediately follow. 

(d) Since $f$ is convex, $\psi^{s}$ in \eqref{eq:psi_apd_def} is
$(1/2)$-strongly convex at every (outer) iteration of \prettyref{alg:pf_apd}.
Consequently, using \prettyref{prop:acg_alg_props}(b) with $\mu=1/2$,
the inputs and outputs given to \prettyref{alg:pf_acg} by \prettyref{alg:pf_apd_line_search},
and the definition of $\psi^{s}$, it follows that every call to \prettyref{alg:pf_apd_line_search}
by \prettyref{alg:pf_apd} stops at (line search) iteration $\ell=0$,
i.e., the conditions in \eqref{eq:gd_spec_ineq}\textbf{ }are satisfied
when they are first checked. Using the update rule in Step~1 of \prettyref{alg:pf_apd}
and the previous conclusion, we have that $m_{k+1}=m_{k}/\alpha$
for every $k\geq0$. Inductively, it then follows that $m_{k}=\alpha^{-k}m_{0}$
for every $k\geq0$. We now prove the claimed complexity bound. In
view of the fact that $\{1/m_{k}\}$ is bounded below from part (a),
let $\overline{K}$ be the smallest index such that $\overline{K}\geq2$
and
\begin{equation}
\sum_{k=0}^{\overline{K}-2}\frac{1}{m_{k+1}}\leq\frac{1}{m_{0}}+\frac{4\theta\alpha^{-1}m_{0}\cdot R_{2m_{0}/\alpha}(z_{0})}{\varepsilon^{2}}\leq\sum_{k=0}^{\overline{K}-1}\frac{1}{m_{k+1}}.\label{eq:cvx_tech_bd}
\end{equation}
Using \eqref{eq:vj_ref_bd} with $\nu=\alpha$, the fact that $m_{1}=m_{0}/\alpha$,
and the same type of arguments as in part (c), we have that
\begin{align*}
\min_{1\leq k\leq\bar{K}-1}\|v_{k+1}\|^{2} & \leq\frac{4\theta m_{1}R_{2m_{1}}(z_{0})}{\sum_{k=1}^{\overline{K}-1}m_{k+1}^{-1}}=\frac{4\theta\alpha^{-1}m_{0}\cdot R_{2m_{0}/\alpha}(z_{0})}{-m_{0}^{-1}+\sum_{k=0}^{\overline{K}-1}m_{k+1}^{-1}}\leq\varepsilon^{2}
\end{align*}
and, hence, the number of outer iterations $K(\varepsilon)$ is bounded
above by $\overline{K}$. It now remains to show that $\overline{K}$
is bounded above by the expression on the right-hand side of \eqref{eq:K_rho_spec_bds}.
Using the identity $m_{k}=\alpha^{-k}m_{0}$ and the right-hand side
of \eqref{eq:cvx_tech_bd}, we have
\[
\frac{1}{m_{0}}+\frac{4\theta\alpha^{-1}m_{0}\cdot R_{2m_{0}/\alpha}(z_{0})}{\varepsilon^{2}}\geq\sum_{k=0}^{\overline{K}-2}\frac{1}{m_{k+1}}=\frac{1}{m_{0}}\sum_{k=0}^{\overline{K}-2}\alpha^{k+1}\geq\frac{\alpha^{\overline{K}-1}}{m_{0}}.
\]
Applying the function $\log_{\alpha}(\cdot)$ to both sides of the
above inequality and re-arranging terms yields the desired bound on
$\overline{K}$.

(e) Using the definition of $v_{k+1}$ and \prettyref{lem:acg_specialization}
with $\psi^{s}$ as in \eqref{eq:psi_apd_def}, we have 
\begin{align*}
v_{k+1} & \in2m_{k+1}\left[\nabla\psi^{s}(z_{k+1})+\pt\psi^{s}(z_{k+1})\right]+2m_{k+1}(z_{k}-z_{k+1})\\
 & =2m_{k+1}\left[\frac{\nabla f(z_{k+1})}{2m_{k+1}}+(z_{k+1}-z_{k})+\frac{\pt h(z_{k+1})}{m_{k+1}}\right]+2m_{k+1}(z_{k}-z_{k+1})\\
 & =\nabla f(z_{k+1})+\pt h(z_{k+1}).
\end{align*}
The fact that the last iterate solves Problem~${\cal CO}$ follows
from the above inclusion and the termination condition in Step~2
of \prettyref{alg:pf_apd}.
\end{proof}

\section{Applications}

\label{sec:extensions} 

This section describes a few possible applications of \prettyref{alg:pf_apd}
in more general optimization frameworks.\medskip{}

\emph{Min-Max Smoothing}. In \cite{kong2021minmax}, a smoothing framework
was proposed for finding $\varepsilon$-stationary points of the nonconvex-concave
min-max problem 
\begin{equation}
\min_{x\in\rn}\max_{y\in\r^{l}}\left[\phi(x,y)+h(x)\right]\label{eq:min_max}
\end{equation}
where $h$ is as in assumption \ref{asm:a1}, $\phi(\cdot,y)$ is
$m_{x}$-weakly convex and differentiable, $-\phi(x,\cdot)$ is proper
closed convex, and $\nabla_{x}\phi(\cdot,\cdot)$ is Lipschitz continuous. 

The framework considers finding an $\varepsilon$-stationary point
of $h$ plus a smooth approximation $\hat{p}$ of $\max_{y\in Y}\phi(\cdot,y)$.
Choosing a special smoothing constant such that the curvature pair
$(\hat{m},\hat{M})$ of $\hat{p}$ satisfies $\hat{m}=m_{x}$ and
$\hat{M}=\Theta(\varepsilon^{-1}D_{y})$ (resp. $\hat{M}=\Theta(D_{y}^{2}\varepsilon^{-2})$),
where $D_{y}$ is diameter of $\dom(-\phi(x,\cdot))$, it was shown
that an $\varepsilon$-stationary point of $\hat{p}$ yields an $\varepsilon$-primal-dual
(resp. directional) stationary point of \eqref{eq:min_max}.

If we use PF.APD with $m_{0}=\varepsilon$ to obtain an $\varepsilon$-stationary
point of $\hat{p}$ as above, then an $\varepsilon$-primal-dual (resp.
directional) stationary point of \eqref{eq:min_max} is obtained in
${\cal O}(\varepsilon^{-2.5})$ (resp. ${\cal O}(\varepsilon^{-3})$)
inner iterations, and this matches, up to logarithmic terms, the complexity
bounds for the smoothing method in \cite{kong2021minmax}. Moreover,
when $\phi(\cdot,y)$ is convex, the above complexity is ${\cal O}(\varepsilon^{-1})$
(resp. ${\cal O}(\varepsilon^{-1.5})$), and this appears to be the
first parameter-free approach that could be used for min-max optimization.
This approach also has the  strong advantage that it does not need
to know $D_{y}$.

\medskip{}

\emph{Penalty Method}. In \cite{kong2020efficient}, a penalty method
is proposed for finding $\varepsilon$-KKT points of the linearly-constrained
nonconvex optimization problem 
\begin{equation}
\min_{x\in\rn}\left\{ \phi(x):=f(x)+h(x):Ax=b\right\} \label{eq:lin-constr_prb}
\end{equation}
where $(f,h)$ are as in \ref{asm:a1}--\ref{asm:a3}. It was shown
that if the penalty method uses an algorithm ${\cal A}$ that needs
${\cal O}(T_{m,M}(\varepsilon))$ iterations to obtain an $\varepsilon$-stationary
point of $\phi$, then the total number of inner iterations of the
penalty method (for finding an $\varepsilon$-KKT point) is ${\cal O}(T_{m,\varepsilon^{-2}}(\varepsilon))$. 

If we use the PF.APD with $m_{0}=\varepsilon$ as algorithm ${\cal A}$
above, then an $\varepsilon$-KKT point of \eqref{eq:lin-constr_prb}
is obtained in ${\cal O}(\varepsilon^{-3})$ inner iterations which
matches the complexity bound for the particular penalty method in
\cite{kong2020efficient} (which uses the AIPP in \cite{kong2019complexity}
for algorithm ${\cal A}$). Moreover, when $f$ is convex, the above
complexity is ${\cal O}(\varepsilon^{-1.5})$. Like in the above discussion
for min-max smoothing, this appears to be the first parameter-free
approach used for linearly-constrained composite optimization. 

\section{Numerical Experiments}

\label{sec:numerical}

This section presents experiments that demonstrate the numerical efficiency
of PF.APD. Comments about the results are given in \prettyref{subsec:numer_comments}.

We first describe the benchmark algorithms, the implementation of
APD, and the computing environment. The benchmark algorithms are instances
of PGD, AIPP, ANCF, and UPF described in \prettyref{sec:intro} and
\prettyref{tab:compl_compare}. Specifically, AIPP uses $\sigma=1/4$,
ANCF uses $\theta=1.25$, and UPF uses $\gamma_{1}=\gamma_{2}=0.4$,
$\gamma_{3}=1$, $\beta_{0}=1$, and $\hat{\lam}_{0}=1$. Moreover,
UPF uses $\hat{\lam}_{k}$ for the initial estimate of $\hat{\lam}_{k+1}$
for $k\geq1$ and AIPP stops its call of ACG when the condition $\|u_{j}\|^{2}+2\eta_{j}\leq\sigma\|x_{0}-x_{j}+u_{j}\|^{2}$
holds (inside of ACG) instead of prescribing a fixed number of ACG
iterations. The implementations for ANCF and UPF were generously provided
by the respective authors of \cite{liang2021fista} and \cite{ghadimi2019generalized},
while the author implemented AIPP and PGD.\footnote{See \texttt{https://github.com/wwkong/nc\_opt/tree/master/tests/papers/apd}
for the source code of the experiments.} Note that we did not consider the VAR-FISTA method in \cite{sim2020fista}
because: (i) its steps were similar to ANCF and (ii) we already had
a readily available and optimized code for the ANCF method.

The implementation of PF.APD, abbreviated as APD, is as in \prettyref{alg:pf_apd}
with $\alpha=\beta=2$, $\rho=1/\sqrt{\alpha}$, $\hat{m}=m_{k}$
for every $k\geq1$, and the following additional updates at the beginning
of every call to \prettyref{alg:pf_acg} and the $(k+1)^{{\rm th}}$
iteration of \prettyref{alg:pf_apd}, respectively:
\begin{equation}
L_{0}\gets\frac{L_{0}}{1+\beta/2},\quad m_{k+1}\gets\max\left\{ m_{0},\frac{m_{k+1}}{1+\alpha/2}\right\} .\label{eq:numer_mod}
\end{equation}
This is done to allow a possible decrease in both of the curvature
estimates. While we do not show convergence of this modified PF.APD,
we believe that convergence can be established using similar techniques
as in \cite{nesterov2013gradient}. It is worth mentioning that the
modification in \eqref{eq:numer_mod} substantially improves upon
the numerical performance of PF.APD compared to the version given
in \prettyref{alg:pf_apd}.

All experiments were run in MATLAB 2023a under a 64-bit Windows 11
machine with an Intel Core i7-10700K processor and 16 GB of RAM. All
algorithms except AIPP use an initial curvature estimate of $(m_{0},M_{0})=(1,1)$,
and each algorithm stops when it finds a pair $(\bar{z},\bar{v})$
solving Problem~$\boldsymbol{{\cal CO}}$ for some $\varepsilon>0$.
A time limit of 1200 (resp. 2400) seconds was prescribed for the problems
in \prettyref{subsec:qsdp} and \ref{subsec:lrmc} (resp. \prettyref{subsec:svr}).
We also set an (innermost) iteration limit of 500000 (resp. 10000)
for \prettyref{subsec:svr} (resp. \prettyref{subsec:lrmc}).

\subsection{Quadratic Semidefinite Programming}

\label{subsec:qsdp}

The problem of interest is the 400-variable nonconvex quadratic semidefinite
programming (QSDP) problem
\begin{align}
\min_{Z\in\r^{35\times35}}\  & -\frac{\eta_{1}}{2}\|D{\cal B}(Z)\|_{2}^{2}+\frac{\eta_{2}}{2}\|{\cal A}(Z)-b\|_{2}^{2},\label{eq:qsdp}\\
\text{s.t.}\  & \text{\ensuremath{\trc}}(Z)=1,\quad Z\in{\cal S}_{+}^{35},\nonumber 
\end{align}
where ${\cal S}_{+}^{n}$ is the $n$-dimensional positive semidefinite
cone, $\trc(Z)$ is the trace of a matrix, $b\in\r^{10}$, $D\in\r^{10\times10}$
is a diagonal matrix with nonzero entries randomly generated from
$\{1,...,1000\}$, $(\eta_{1},\eta_{2})\in\r_{++}^{2}$ are chosen
to yield a particular curvature pair, and ${\cal A},{\cal B}:{\cal S}_{+}^{20}\mapsto\r^{10}$
are linear operators defined by 
\[
[{\cal A}(Z)]_{j}=A_{j}\bullet Z,\quad[{\cal B}(Z)]_{j}=B_{j}\bullet Z
\]
for matrices $\{A_{j}\}_{j=1}^{10},\{B_{j}\}_{j=1}^{10}\subseteq\r^{20\times20}$.
Moreover, the entries in these matrices and $b$ were sampled from
the uniform distribution on $[0,1]$. 

To build the decomposition in \eqref{eq:main_prb}, we set $f$ equal
to the objective function of \eqref{eq:qsdp}, $h$ equal to the indicator
function of the constraint set of \eqref{eq:qsdp}. The starting point
was set to $z_{0}=I_{20}/20,$ where $I_{20}$ is an identity matrix,
and the tolerance was set to $\varepsilon=10^{-6}(1+\|\nabla f(z_{0})\|_{2}).$ 

\renewcommand{\arraystretch}{0.8}\setlength\tabcolsep{4pt}
\begin{table}[tbh]
\begin{centering}
{\tiny{}}%
\begin{tabular}{c|cccc|cccc|cccc}
\multicolumn{1}{c|}{} & \multicolumn{4}{c|}{{\tiny{}\# of Function Evaluations}} & \multicolumn{4}{c|}{{\tiny{}\# of Gradient Evaluations}} & \multicolumn{4}{c}{{\tiny{}Runtime (seconds)}}\tabularnewline
\hline 
{\tiny{}$m,M$} & {\tiny{}UPF} & {\tiny{}ANCF} & {\tiny{}AIPP} & {\tiny{}APD} & {\tiny{}UPF} & {\tiny{}ANCF} & {\tiny{}AIPP} & {\tiny{}APD} & {\tiny{}UPF} & {\tiny{}ANCF} & {\tiny{}AIPP} & {\tiny{}APD}\tabularnewline
\hline 
{\tiny{}$10^{2},10^{4}$} & {\tiny{}6.5E4} & {\tiny{}2.1E4} & {\tiny{}7.1E4} & \textbf{\tiny{}1.1E3} & {\tiny{}1.3E4} & {\tiny{}1.6E4} & {\tiny{}6.7E4} & \textbf{\tiny{}2.1E3} & {\tiny{}9.2E1} & {\tiny{}2.7E1} & {\tiny{}1.1E2} & \textbf{\tiny{}3.2E0}\tabularnewline
{\tiny{}$10^{2},10^{5}$} & {\tiny{}1.9E5} & {\tiny{}4.4E4} & {\tiny{}4.1E5} & \textbf{\tiny{}3.3E3} & {\tiny{}3.8E4} & {\tiny{}3.3E4} & {\tiny{}3.9E5} & \textbf{\tiny{}6.7E3} & {\tiny{}2.6E2} & {\tiny{}5.8E1} & {\tiny{}6.5E2} & \textbf{\tiny{}9.9E0}\tabularnewline
{\tiny{}$10^{2},10^{6}$} & {\tiny{}3.0E5} & {\tiny{}5.9E4} & {\tiny{}7.6E5} & \textbf{\tiny{}7.1E3} & {\tiny{}6.1E4} & {\tiny{}4.4E4} & {\tiny{}7.0E5} & \textbf{\tiny{}1.4E4} & {\tiny{}4.3E2} & {\tiny{}7.9E1} & {\tiny{}1.2E3} & \textbf{\tiny{}2.1E1}\tabularnewline
{\tiny{}$10^{3},10^{7}$} & {\tiny{}3.0E5} & {\tiny{}5.9E4} & {\tiny{}7.6E5} & \textbf{\tiny{}1.0E4} & {\tiny{}6.1E4} & {\tiny{}4.4E4} & {\tiny{}6.9E5} & \textbf{\tiny{}2.0E4} & {\tiny{}4.3E2} & {\tiny{}8.1E1} & {\tiny{}1.2E3} & \textbf{\tiny{}3.0E1}\tabularnewline
{\tiny{}$10^{2},10^{7}$} & {\tiny{}3.3E5} & {\tiny{}6.6E4} & {\tiny{}2.6E5} & \textbf{\tiny{}1.2E4} & {\tiny{}6.5E4} & {\tiny{}5.0E4} & {\tiny{}1.3E5} & \textbf{\tiny{}2.4E4} & {\tiny{}4.5E2} & {\tiny{}8.6E1} & {\tiny{}2.5E2} & \textbf{\tiny{}3.4E1}\tabularnewline
{\tiny{}$10^{1},10^{7}$} & {\tiny{}5.8E5} & {\tiny{}1.4E5} & {\tiny{}8.8E4} & \textbf{\tiny{}2.0E4} & {\tiny{}1.2E5} & {\tiny{}1.1E5} & {\tiny{}4.4E4} & \textbf{\tiny{}4.1E4} & {\tiny{}7.9E2} & {\tiny{}1.9E2} & {\tiny{}8.3E1} & \textbf{\tiny{}5.8E1}\tabularnewline
\end{tabular}{\tiny{}\smallskip{}
}{\tiny\par}
\par\end{centering}
\caption{Unique function evaluations, unique gradient evaluations, and runtimes
in the QSDP experiments for different curvature pairs $(m,M)$. The
bolded numbers indicate the best algorithm in terms of the number
of evaluations (less is better) and runtime (less is better). Entries
marked with ``-'' are those that did not terminate within the prescribed
time limit.\label{tab:qsdp}}
\end{table}
\prettyref{tab:qsdp} reports the number of unique function evaluations,
unique gradient evaluations, and runtime (in seconds) for different
curvature pairs $(m,M)$, and \prettyref{fig:qsdp} plots the minimum
norm of the normalized stationarity residual $\|\bar{v}\|$ over iteration
count for each algorithm and curvature pairs $(m,M)=$ $(10^{2},10^{4})$,
$(10^{2},10^{5})$, and $(10^{2},10^{6})$.
\begin{figure}[tbh]
\begin{centering}
\includegraphics[viewport=0bp 0bp 1500bp 353bp,width=1\textwidth,trim=3.0cm 0cm 3.0cm 0cm,clip]{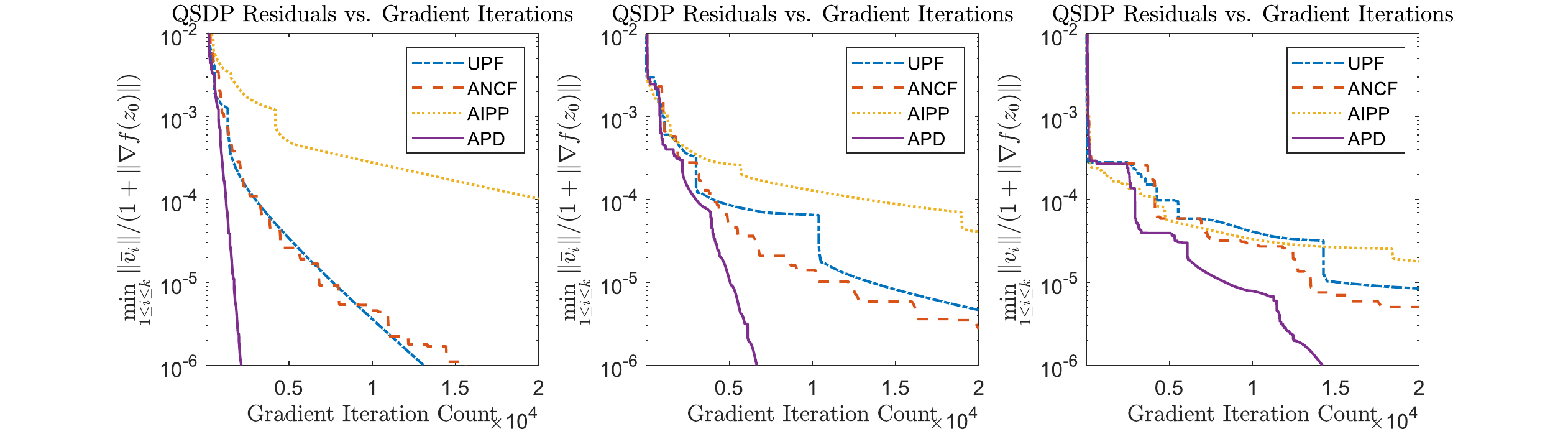}
\par\end{centering}
\caption{Plots of the minimum norm of the normalized stationarity residual
$\|\bar{v}\|$ over iteration count in the QSDP experiments. The curvature
pairs for the plots are $(10^{2},10^{4})$, $(10^{2},10^{5})$, and
$(10^{2},10^{6})$ from left-to-right.\label{fig:qsdp}}

\end{figure}

\subsection{Sparse Vector Recovery}

\label{subsec:svr}

The problem of interest is the penalized sparse vector recovery (SVR)
problem \cite{wen2018survey}

\begin{equation}
\min_{z\in\rn}\frac{1}{2}\|Az-b\|_{2}^{2}+\frac{\tau}{2}\|z\|_{2}^{2}+{\rm LPL}_{\gamma,\delta}(\|z\|_{2})\label{eq:prb_svr}
\end{equation}
where $\tau=10^{-2}$, $A\in\r^{\ell\times p}$ with $\ell\geq p$,
$b=A\tilde{u}$ where $u$ is a random vector whose entries are sampled
uniformly from {[}0,1{]}, for $(\gamma,\delta)=(10,10^{-1})$, the
function ${\rm LPL}_{\gamma,\delta}(z)=\gamma[1-\exp(-z/\delta)]$
is the concave Laplace penalty function \cite{trzasko2008highly}
at $z$. The goal of this problem is to find a sparse vector $\hat{z}$
such that $A\hat{z}$ is close to $b$. 

Each matrix $A$ is built from a recommender dataset where each entry
corresponds to a user-item rating. Specifically, the datasets were
taken from the well-known Jester, MovieLens 100K, and FilmTrust datasets
and the musical instruments and patio, lawn, and garden products Amazon
Review datasets published by the University of California San Diego.
The dimensions $(\ell,p)$ of each matrix generated by the previous
datasets were $(24938,100)$, $(9724,610)$, $(2071,1508)$, $(1429,900)$,
$(1686,962)$, respectively.

To put \eqref{eq:prb_svr} into the form of \eqref{eq:main_prb},
we use the decomposition given in \cite{yao2016efficient} where $h$
is a multiple of the 1-norm and $f$ is the function in \eqref{eq:prb_svr}
minus $h$. The starting point $z_{0}$ was set to be a vector whose
entries are all equal to $p$, and the tolerance was set to $\varepsilon=10^{-10}(1+\|\nabla f(z_{0})\|_{2}).$
Following the analysis in \cite{yao2016efficient}, AIPP uses the
curvature pair $(m,M)=(2\gamma/\delta^{2},\tau+\sigma_{\max}^{2}(A))$,
where $\sigma_{\max}(A)$ is the largest singular value of $A$. 

\begin{table}[tbh]
\begin{centering}
{\tiny{}}%
\begin{tabular}{c|cccc|cccc|cccc}
\multicolumn{1}{c|}{} & \multicolumn{4}{c|}{{\tiny{}\# of Function Evaluations}} & \multicolumn{4}{c|}{{\tiny{}\# of Gradient Evaluations}} & \multicolumn{4}{c}{{\tiny{}Runtime (seconds)}}\tabularnewline
\hline 
{\tiny{}$\ell,p$} & {\tiny{}UPF} & {\tiny{}ANCF} & {\tiny{}AIPP} & {\tiny{}APD} & {\tiny{}UPF} & {\tiny{}ANCF} & {\tiny{}AIPP} & {\tiny{}APD} & {\tiny{}UPF} & {\tiny{}ANCF} & {\tiny{}AIPP} & {\tiny{}APD}\tabularnewline
\hline 
{\tiny{}1429, 900} & {\tiny{}6.9E3} & {\tiny{}3.5E3} & {\tiny{}1.5E4} & \textbf{\tiny{}3.7E2} & {\tiny{}8.0E2} & {\tiny{}2.6E3} & {\tiny{}1.1E4} & \textbf{\tiny{}7.3E2} & {\tiny{}2.7E0} & {\tiny{}1.2E0} & {\tiny{}1.1E1} & \textbf{\tiny{}3.4E-1}\tabularnewline
{\tiny{}1686, 962} & {\tiny{}2.9E4} & {\tiny{}1.1E4} & {\tiny{}7.7E4} & \textbf{\tiny{}2.6E3} & {\tiny{}4.9E3} & {\tiny{}8.2E3} & {\tiny{}5.8E4} & \textbf{\tiny{}3.8E3} & {\tiny{}1.3E1} & {\tiny{}4.1E0} & {\tiny{}6.0E1} & \textbf{\tiny{}2.4E0}\tabularnewline
{\tiny{}9724, 610} & {\tiny{}3.9E4} & {\tiny{}4.3E4} & {\tiny{}6.2E4} & \textbf{\tiny{}3.2E3} & {\tiny{}6.3E3} & {\tiny{}3.2E4} & {\tiny{}3.3E4} & \textbf{\tiny{}6.2E3} & {\tiny{}3.6E1} & {\tiny{}3.5E1} & {\tiny{}8.4E1} & \textbf{\tiny{}6.0E0}\tabularnewline
{\tiny{}24938, 100} & {\tiny{}5.7E5} & {\tiny{}2.4E5} & {\tiny{}9.8E5} & \textbf{\tiny{}2.5E4} & {\tiny{}1.1E5} & {\tiny{}1.8E5} & {\tiny{}5.0E5} & \textbf{\tiny{}4.8E4} & {\tiny{}1.7E2} & {\tiny{}5.0E1} & {\tiny{}4.3E2} & \textbf{\tiny{}1.6E1}\tabularnewline
{\tiny{}2071, 1508} & {\tiny{}-} & {\tiny{}2.9E5} & {\tiny{}-} & \textbf{\tiny{}2.8E4} & {\tiny{}-} & {\tiny{}2.2E5} & {\tiny{}-} & \textbf{\tiny{}5.5E4} & {\tiny{}-} & {\tiny{}1.3E3} & {\tiny{}-} & \textbf{\tiny{}2.6E2}\tabularnewline
\end{tabular}{\tiny{}\smallskip{}
}{\tiny\par}
\par\end{centering}
\caption{Unique function evaluations, unique gradient evaluations, and runtimes
in the SVR experiments for different datasets and their dimensions
$(\ell,p)$. The bolded numbers indicate the best algorithm in terms
of the number of evaluations (less is better) and runtime (less is
better). Entries marked with ``-'' are those that did not terminate
within the prescribed time or iteration limit.\label{tab:svr}}
\end{table}
\prettyref{tab:svr} reports the unique function evaluations, unique
gradient evaluations, and runtime (in seconds) for the different datasets
mentioned above, and \prettyref{fig:svr} plots the minimum norm of
the normalized stationarity residual $\|\bar{v}\|$ over the gradient
count for each algorithm and the first, second, and fourth row of
\prettyref{tab:svr}.
\begin{figure}[tbh]
\begin{centering}
\includegraphics[width=1\textwidth,trim=3.0cm 0cm 3.0cm 0cm,clip]{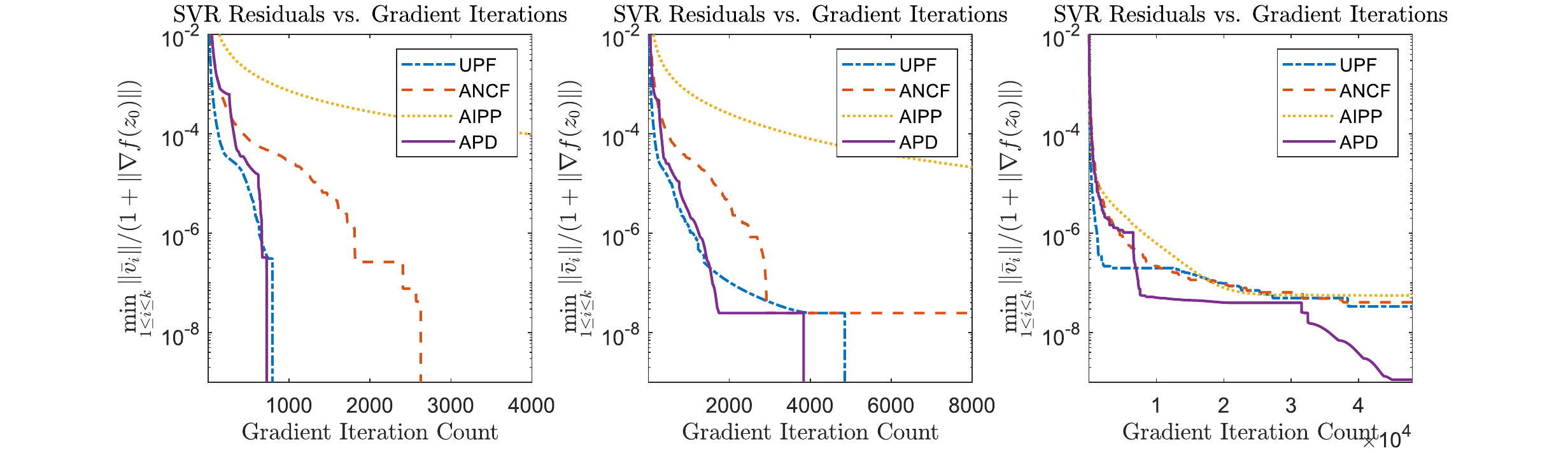}
\par\end{centering}
\caption{Plots of the minimum norm of the normalized stationarity residual
$\|\bar{v}\|$ over iteration count in the SVR experiments. The dimensions
and upper curvature $(\ell,p)$ for the plots are $(1429,900)$, $(1686,962)$,
and $(24938,100)$ from left-to-right.\label{fig:svr}}
\end{figure}

\subsection{Low-Rank Matrix Completion}

\label{subsec:lrmc}

The problem of interest is the penalized nonconvex low-rank matrix
completion (LRMC) problem \cite{wen2018survey,yao2016efficient}

\begin{equation}
\min_{Z\in\r^{\ell\times p}}\frac{1}{2}\|\Pi_{\Omega}(Z)-\Pi_{\Omega}(X)\|_{F}^{2}+\frac{\tau}{2}\|Z\|_{F}^{2}+({\rm MCP}_{\gamma,\delta}\circ\sigma)(Z),\label{eq:prb_mcp}
\end{equation}
where $\tau=10^{-7}$, $X\in\r^{\ell\times p}$ is a reference image,
$\sigma:\r^{\ell\times p}\mapsto\r^{\min\{\ell,p\}}$ maps a matrix
to its vector of singular values, for $(\gamma,\delta)=(450,10^{-4})$
the function ${\rm MCP}_{\gamma,\delta}(z)$ is the minimax concave
penalty (MCP) function \cite{zhang2010nearly} at $z$ (which takes
value $\gamma z-z^{2}/(2\delta)$ if $z\leq\gamma\delta$ and $\gamma^{2}\delta/2$
otherwise), and, for a given corrupted image $\Omega$, the function
$\Pi_{\Omega}:\r^{\ell\times p}\mapsto\r^{\ell\times p}$ is the projection
operator that zeros out entries of its input where the corresponding
entry in $\Omega$ is zero. The goal of this problem is to fill in
the zero entries of a corrupted image $\Omega$ of $X$ so that the
resulting image $\hat{Z}$ is close to $X$. 

To put \eqref{eq:prb_mcp} into the form of \eqref{eq:main_prb},
we use the decomposition given in \cite{yao2016efficient} where $h$
is a multiple of the nuclear norm and $f$ is the function in \eqref{eq:prb_mcp}
minus $h$. Experiments were run on different reference images $X$
given in the first row of \prettyref{fig:img_results} and $\Omega$
was set to be a corrupted version of $X$ where we add Gaussian noise
with a 100 dB signal-to-noise ratio and remove 30\% of the resulting
pixels. For illustration, two corrupted images can be found in the
first columns of the last two rows in \prettyref{fig:img_results}.
The starting point $Z_{0}$ was set to be a matrix whose entries were
equal to the average of the grayscale value of $\Omega$, and the
tolerance was set to $\varepsilon=10^{-10}(1+\|\nabla f(Z_{0})\|_{F}).$
Following the analysis in \cite{yao2016efficient}, AIPP uses the
curvature pair $(m,M)=(2/\delta,1+\tau)$.

\begin{figure}[tbh]
\begin{centering}
{\scriptsize{}Reference images with (left-to-right) IDs $=35008,41004,68077,271031,310007$}{\scriptsize\par}
\par\end{centering}
\begin{centering}
\includegraphics[width=1\textwidth]{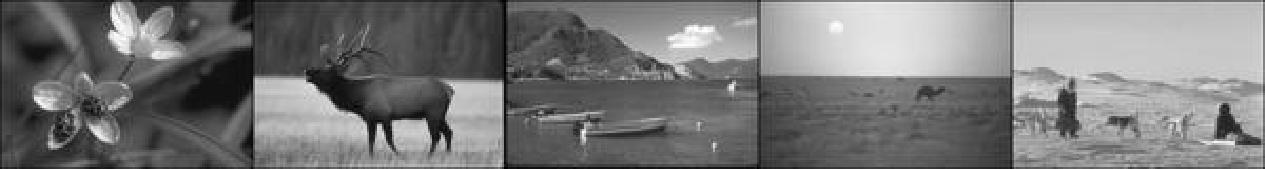}
\par\end{centering}
\begin{centering}
{\scriptsize{}Results of ID $=35008$}{\scriptsize\par}
\par\end{centering}
\begin{centering}
\includegraphics[width=1\textwidth]{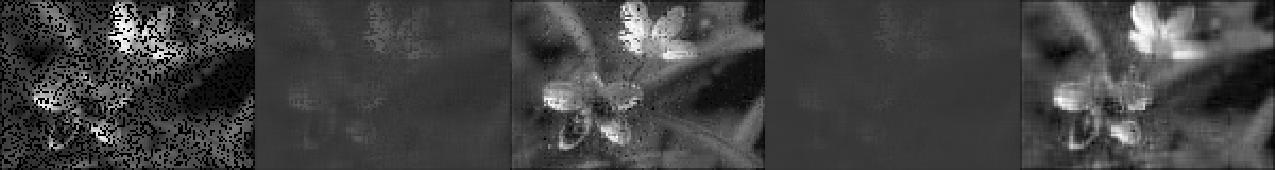}
\par\end{centering}
\begin{centering}
{\scriptsize{}Results of ID $=310007$}{\scriptsize\par}
\par\end{centering}
\begin{centering}
\includegraphics[width=1\textwidth]{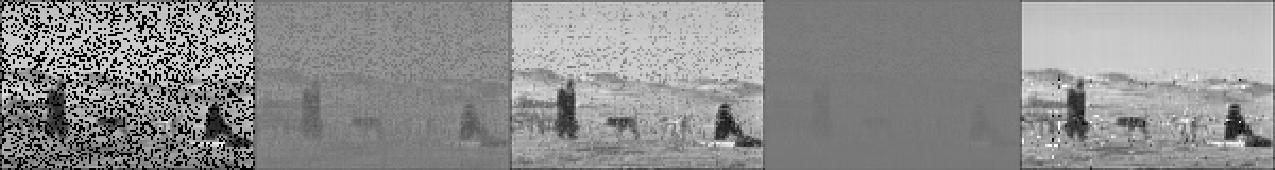}{\tiny{}\smallskip{}
}{\tiny\par}
\par\end{centering}
\caption{The first row presents the downscaled (80$\times$120) reference images
$X$ taken from the Berkeley Segmentation Dataset, along with their
image IDs (in order). The second and third rows present the results
of the LRMC experiments for two of the images. Specifically, each
of these rows presents (from left to right) the corrupted image $\Omega$
and the images generated by UPF, ANCF, AIPP, and APD, respectively.\label{fig:img_results}}

\end{figure}
\prettyref{tab:lrmc} presents the relative error\footnote{For a candidate image $\hat{Z}$, this quantity is defined as $\|\hat{Z}-X\|_{F}$
divided by $\max_{Z\in\Xi}\|Z-X\|_{F}$ where $\Xi$ is the set of
all grayscale images. Its value can range from 0.0 (full recovery)
to 1.0.} of the final candidate image and runtime (in seconds) for the different
reference images, and the last two rows in \prettyref{fig:img_results}
show the candidate images generated by each method for two of the
reference images. 
\begin{table}[tbh]
\begin{centering}
{\tiny{}}%
\begin{tabular}{c|cccc|cccc}
 & \multicolumn{4}{c|}{{\tiny{}Relative Error}} & \multicolumn{4}{c}{{\tiny{}Runtime (seconds)}}\tabularnewline
\hline 
{\tiny{}image id} & {\tiny{}UPF} & {\tiny{}ANCF} & {\tiny{}AIPP} & {\tiny{}APD} & {\tiny{}UPF} & {\tiny{}ANCF} & {\tiny{}AIPP} & {\tiny{}APD}\tabularnewline
\hline 
{\tiny{}35008} & {\tiny{}0.220} & {\tiny{}0.059} & {\tiny{}0.241} & \textbf{\tiny{}0.034} & {\tiny{}104.7} & {\tiny{}174.5} & {\tiny{}89.2} & \textbf{\tiny{}44.4}\tabularnewline
{\tiny{}41004} & {\tiny{}0.259} & {\tiny{}0.103} & {\tiny{}0.312} & \textbf{\tiny{}0.072} & {\tiny{}114.6} & {\tiny{}175.9} & {\tiny{}90.6} & \textbf{\tiny{}45.9}\tabularnewline
{\tiny{}68077} & {\tiny{}0.238} & {\tiny{}0.075} & {\tiny{}0.276} & \textbf{\tiny{}0.046} & {\tiny{}107.6} & {\tiny{}175.6} & {\tiny{}89.2} & \textbf{\tiny{}43.0}\tabularnewline
{\tiny{}271031} & {\tiny{}0.272} & {\tiny{}0.146} & {\tiny{}0.363} & \textbf{\tiny{}0.079} & {\tiny{}117.5} & {\tiny{}176.8} & {\tiny{}95.7} & \textbf{\tiny{}48.0}\tabularnewline
{\tiny{}310007} & {\tiny{}0.265} & {\tiny{}0.079} & {\tiny{}0.324} & \textbf{\tiny{}0.048} & {\tiny{}116.0} & {\tiny{}186.4} & {\tiny{}92.7} & \textbf{\tiny{}44.9}\tabularnewline
\end{tabular}{\tiny{}\smallskip{}
}{\tiny\par}
\par\end{centering}
\caption{Relative errors and runtimes in the LRMC experiments for different
reference images in the LRMC experiments. The bolded numbers indicate
the best algorithm in terms of the relative error (less is better)
and runtime in seconds (less is better).\label{tab:lrmc}}
\end{table}

\subsection{Comments about the numerical results}

\label{subsec:numer_comments}

In \prettyref{subsec:qsdp}, APD substantially outperformed\footnote{5-20x (resp. 2-7x) fewer function (resp. gradient) evaluations for
ANCF and 27-60x (resp. 2-6x) fewer for UPF.} its competitors and its non-adaptive variant AIPP under the given
numerical tolerance $\varepsilon$. However, \prettyref{fig:qsdp}
showed that ANCF was more comparable to PF.APD when the curvature
ratio $M/m$ was large or a larger (more lenient) tolerance was given.
In \prettyref{subsec:svr}, APD consistently outperformed its competitors
on all metrics. For the number of gradient evaluations, UPF performed
similarly to APD but was among the worst adaptive methods for function
evaluations. In \prettyref{subsec:lrmc}, APD generated higher-quality
candidate images compared to its competitors under a fixed iteration
budget. Specifically, it was shown in \prettyref{fig:img_results}
that PF.APD generated images with fewer artifacts, more consistent
lighting, and in a more timely manner. 

\section{Concluding Remarks}

\label{sec:concluding}

This paper establishes iteration complexity bounds for PF.APD that
are only optimal, up to logarithmic terms, in terms of $(M,\Delta_{0},\varepsilon)$
when $f$ is convex and in terms of $(m,M,\Delta_{0},\varepsilon)$
when $f$ is weakly-convex. Consequently, it remains to be seen whether
an optimal complexity bound in terms of $d_{0}$ exists for a parameter-free
and convexity-unaware method. 

To alleviate the issues regarding the $d_{0}$-suboptimal complexity
of APD (specifically, when $f$ is convex and $d_{0}$ is unknown)
one could consider running running $S+1$ instances of PF.APD (either
in lockstep or in parallel) with different initial estimates $m_{0}=1,\varepsilon,\varepsilon/2,\ldots,\varepsilon/2^{S-1}$;
in particular, the whole scheme stops when one of these instances
stops successfully. The number of resolvent evaluations of this approach
is at most $S+1$ times the minimum of the bound in \eqref{eq:apd_spec_compl}
over the different values of $m_{0}$. Consequently, following the
remarks at the end of \prettyref{sec:pf_algs}, if $d_{0}\leq2^{S-1}$
then one of the $S+1$ instances obtains the lower bound in \prettyref{tab:compl_compare}
for the convex case; otherwise, the bound for APD in \prettyref{tab:compl_compare}
is obtained. Moreover, if $S$ is chosen small compared to the other
terms in \eqref{eq:apd_spec_compl} and $d_{0}\leq2^{S-1}$, then
the cost is on the same order of magnitude as the $(M,\Delta_{0},d_{0},\varepsilon)$-complexity
optimal method described at the end of \prettyref{sec:pf_algs} (which
requires knowledge of $d_{0}$).

In addition to the applications in \prettyref{sec:extensions}, it
would be interesting to see if PF.APD could be leveraged to develop
a parameter-free proximal augmented Lagrangian method, following schemes
similar to ones as in \cite{melo2020iteration,kong2020iteration}.

\bibliographystyle{siamplain}
\bibliography{apd_refs}

\end{document}